\documentclass{article}

\usepackage{amsthm}
\usepackage{amsfonts}
\usepackage{amsmath}
\usepackage{mathtools}
\usepackage{graphicx}
\usepackage{xspace}
\usepackage{url}
\usepackage{amssymb, xcolor}
\usepackage{enumerate}
\usepackage{enumitem}

\newtheorem{Theorem}{Theorem}
\newtheorem{Proposition}[Theorem]{Proposition}
\newtheorem{Corollary}[Theorem]{Corollary}
\newtheorem{Lemma}[Theorem]{Lemma}
\newtheorem{Example}{Example}
\newtheorem{Remark}{Remark}
\newtheorem{Definition}{Definition}
\newtheorem{Assumption}{Assumption}

\newcommand{\bd}{\begin{displaymath}}
\newcommand{\ed}{\end{displaymath}}
\newcommand{\be}{\begin{equation}}
\newcommand{\ee}{\end{equation}}

\newcommand{\Uad}{\mathcal{U}}
\newcommand{\Pb}{\mbox{\rm (P)}\xspace}

\newcommand{\dx}{\,\mathrm{d}x}
\newcommand{\dt}{\,\mathrm{d}t}

\newcommand{\llangle}{\big\langle}
\newcommand{\rrangle}{\big\rangle}
\newcommand{\reff}{\eqref}

\numberwithin{equation}{section}
\title{On the solution stability of parabolic optimal control problems\thanks{The first author and the second author were supported by the Austrian Science Foundation (FWF) under grant No I4571.}}

\author{Alberto Dom\'inguez Corella\thanks{Institute of Statistics and Mathematical Methods in Economics, Vienna University of Technology, Austria, {\tt alberto.corella@tuwien.ac.at}}
\and Nicolai Jork\thanks{Institute of Statistics and Mathematical Methods in Economics, Vienna University of Technology, Austria, {\tt nicolai.jork@tuwien.ac.at}}
\and Vladimir M. Veliov\thanks{Institute of Statistics and Mathematical Methods in Economics, Vienna University of Technology, Austria, {\tt vladimir.veliov@tuwien.ac.at}}}

\begin{document}

\maketitle

\begin{abstract}
The paper investigates stability properties of solutions of optimal control problems 
for semilinear parabolic partial differential equations. Hölder or Lipschitz dependence of the optimal solution 
on perturbations are obtained for problems in which the equation and 
the objective functional are affine with respect to the control. The perturbations may appear in both the equation and in the 
objective functional and may nonlinearly depend on the state and control variables. 
The main results are based on an extension of recently introduced assumptions on the
joint growth of the first and second variation of the objective functional. The stability of the optimal solution is obtained
as a consequence of a more general result obtained in the paper -- the proved metric subregularity of the mapping associated with the system of 
first order necessary optimality conditions. This property also enables error estimates for approximation methods.
Lipschitz estimate for the dependence of the optimal control on the Tikhonov regularization 
parameter is obtained as a by-product.
\end{abstract}

\pagestyle{myheadings} \thispagestyle{plain} \markboth{A.~DOM\'{I}NGUEZ CORELLA, N.~JORK AND Vladimir M. Veliov}{}

\section{Introduction} \label{SIntro}

Let $\Omega\subset\mathbb R^n$, $1\le n\le 3$, be a bounded domain with Lipschitz boundary $\partial \Omega$. 
For a finite $T>0$, denote by $Q:=\Omega\times (0,T)$ the space-time cylinder and by 
$\Sigma:=\partial \Omega \times (0,T)$ its lateral boundary.
In the present paper, we investigate the following optimal control problem:
	\begin{equation}
 		  \Pb \  \ \min_{u \in \Uad}\bigg\{ J(u) := \int_Q L(x,t,y(x,t),u(x,t))\dx\dt\bigg\},
	\label{ocp1}
	\end{equation}
subject to
	\begin{equation}
		\left\{ \begin{array}{lll}
		\frac{\partial y}{\partial t}+\mathcal Ay+f(\cdot,y)& =u\  &\text{ in }\ Q,\\
		y=0 \text{ on } \Sigma,\quad y(\cdot,0)& =  y_0\ &\text{ on } \Omega.\\
	\end{array} \right.
	\label{see1}
	\end{equation}
Denote by $y_u$ the unique solution to the semilinear parabolic equation \eqref{see1} that corresponds to control 
$u\in L^r(Q)$, where $r$ is a fixed number satisfying the inequality $r>1+\frac{n}{2}$. 
For functions $u_a, u_b\in L^\infty(Q)$ such that $ u_{a} < u_{b} $ a.e in $ Q$, 
the set of feasible controls is given by 
	\begin{equation}
		\Uad := \{u \in L^\infty(Q) \vert \ u_a \le u \le u_b\ \text{ for a.a. } (x,t) \in Q\}.
	\label{const}
	\end{equation}
The objective integrand in \reff{ocp1} is defined as
	\begin{equation} \label{genobj}
		L(x,t,y,u):=L_0(x,t,y)+(my+g)u,
	\end{equation}
where $m$ is a number, $g$ is a function in $L^\infty(Q)$ and $L_0$ satisfies appropriate smoothness condition
(see Assumption \ref{exist.2} in Subsection \ref{SSprel}).

The goal of the present paper is to obtain stability results for the optimal solution of problem \reff{ocp1}--\reff{const}.
The meaning of ``stability'' we focus on, is as follows. Given a reference optimal control $\bar u$ 
and the corresponding solution $y_{\bar u}$,
the goal is to estimate the distance (call it $\Delta$) from the optimal solutions $(u,y_u)$ of a disturbed version of problem 
\reff{ocp1}--\reff{const} to the pair $(\bar u, y_{\bar u})$, in terms of the size of the perturbations 
(call it $\delta$). The perturbations may enter either in the objective integrand
or in the state equation, and the meaning of ``distance'' and ``size'' in the previous sentence will be clarified in the sequel
in terms of appropriate norms. If an estimation $\Delta \leq {\rm const.} \delta^\theta$ holds with $\theta \in (0,1)$, we talk 
about {\em H\"older stability}, while in the case $\theta = 1$ we have {\em Lipschitz stability}. 

A powerful technique for establishing stability properties of the solutions of optimization problems is based on regularity properties 
of the system of first order necessary optimality conditions (see e.g. \cite{DR}). 
In the case of problem \reff{ocp1}--\reff{const}, these are represented by a {\em differential variational inequality} (see e.g.
\cite{CDKV-17,PSS}), consisting of two parabolic equations (the primal equation \reff{ocp1} and the corresponding adjoint
equation) and one variational inequality representing the condition for minimization of the Hamiltonian associated with the problem.
The Lipschitz or H\"older stability of the solution of problem \reff{ocp1}--\reff{const} is then a consequence of the property 
of {\em metric subregularity} (see \cite{CDK,DR}) of the mapping defining this differential variational inequality.  
An advantage of this approach is that it unifies 
in a compact way the study of 
stability of optimal solutions under a variety of perturbations (linear or nonlinear). 
Therefore, the main result in the present paper focuses on conditions for metric subregularity 
of the mapping associated with the 
first order optimality conditions for problem \reff{ocp1}--\reff{const}. 
These conditions are related to appropriate second order sufficient optimality conditions, which  are revisited and extended 
in the paper. Several results for stability of the solutions are obtained as a consequence.

The commonly used second order sufficient optimality conditions for ODE or PDE optimal control problems involve a
{\em coercivity condition}, requiring strong positive definiteness of the objective functional as a function of the control in a Hilbert space.
We stress that problem \reff{ocp1}--\reff{const} is affine with respect to the control variable and such a coercivity condition
is not fulfilled. The theory of sufficient optimality theory and the regularity theory for affine optimal control of ODE systems 
have been developed in the past decade,
see \cite{OSVE} and the bibliography therein. Sufficient conditions for weak or strong local optimality for optimal control problems with constraints given by elliptic or parabolic equations 
are developed in \cite{CDJ2022,Casas-Mateos2020,Casas12, CMR,CT16,CWW, DJV2022}.
A detailed discussion thereof is provided in Section \ref{CS}. In contrast with the elliptic setting, there 
are only a few stability results for semilinear parabolic optimal control problems. Progress in this regard for a tracking type objective
functional was made for instance in \cite{CRT15,CT16} where stability with respect to perturbations in the objective functional was studied, 
and in \cite{CT22}, where stability with respect to perturbations in the initial data was 
investigated. 
We mention that for a linear state 
equation and a tracking type objective functional, Lipschitz estimates were obtained in \cite{VD18}
under an additional assumption on the structure of the 
optimal control. More comprehensive discussion about the sufficiency theory and stability can be found in Section \ref{SProblem}.

The main novelty in the present paper is the study of the subregularity property of the optimality mapping associated with problem 
 \reff{ocp1}--\reff{const}.  In contrast with the case of coercive problems, our assumptions in the affine case
jointly involve the first and the second order variations of the objective functional
with respect to the control. These assumptions are weaker than the ones in the existing literature in the context of
sufficient optimality conditions, however, they are strong enough to imply metric subregularity of the optimality mapping.
The subregularity result is used to obtain new H\"older- and Lipschitz estimates for the solution 
of the considered optimal control problem. An error estimate for the Tikhonov regularization is obtained as a consequence.
\\

The  obtained subregularity result provides a base for convergence and error analysis for discretization methods 
applied to problem \reff{ocp1}--\reff{const}. The point is, that numerical solutions of the discretized versions of the problem
typically satisfy approximately first order optimality conditions for the discretized problem and after appropriate embedding 
in the continuous setting  \reff{ocp1}--\reff{const}, satisfy the optimality conditions 
for the latter problem with a residual
depending on the approximation and the discretization error. Then the subregularity property of the optimality mapping associated
with  \reff{ocp1}--\reff{const} provides an error estimate. Notice that the (Lipschitz) stability of the solution alone is not enough 
for such a conclusion, and this is an important motivation for studying subregularity of the optimality mapping
rather than only stability of the solutions. However, we do not go into this subject, 
postponing it to a later paper based on the present one.   \\

The paper is organized as follows.
The analysis of the optimal control problem \reff{ocp1}--\reff{const} begins in  Section \ref{SProblem}.
We recall the state of the art regarding second order sufficient conditions for weak and strong (local) optimality, as well as known sufficient conditions for
stability of optimal controls and states under perturbations.  In Section \ref{SUni} we formulate and discuss the assumptions
on which our further analysis on sufficiency and stability is based. 
The strong subregularity of the optimality mapping is proved in Section~\ref{SSR}.  
In Section \ref{Stability}, we obtain stability results for the optimal control problem under non-linear perturbations, postponing some technicalities to Assumption \ref{secA1}. 
Finally, we support the theoretical results with some examples.


\subsection{Preliminaries} \label{SSprel}

We begin with some basic notations and definitions.
Given a non-empty, bounded and Lebesgue measurable set $X\subset \mathbb R^n$, 
we denote by $L^p(X)$, $1\leq p\leq\infty$, the Banach spaces of all measurable 
functions $f : X \to  \mathbb R$ for which the usual norm $\| f\|_{L^p(X)}$ is finite.
For a bounded Lipschitz domain $X\subset \mathbb R^n$ (that is, a set with Lipschitz boundary), 
the Sobolev space $H^1_0(X)$ consists of functions that vanish on the boundary (in the trace sense)
and that have weak first order derivatives in $L^2(X)$.
The space $H_0^1(X)$ is equipped with its usual norm denoted by $\|\cdot\|_{H^1_0(X)}$. By $ H^{-1}(X)$ we denote the topological dual of $H_0^1(X)$,
equipped with the standard norm $\|\cdot\|_{H^{-1}(X)}$.
Given a real Banach space $Z$, the space $L^p(0,T\text{; } Z)$ consist of all strongly measurable functions 
$y:[0,T]\to Z$ that satisfy 
\begin{align*}
		\|y\|   _{L^p(0,T\text{; }Z)}:=\Big(\int_0^T \|   y(t) \|_{Z}^p\dt\Big)^{\frac{1}{p}}<\infty
                      \qquad \mbox{if } \; 1\leq p<\infty,
\end{align*}
or, for $p=\infty$,
\begin{align*}
		\|y\|_{L^\infty(0,T\text{; }Z)}:=\text{inf}\{M\in\mathbb{R}\ \vert\ \|y(t)\|_{Z}\leq M \text{ a.e } t\in (0,T)\}<\infty.
\end{align*}
The Hilbert space $W(0,T)$ consists of all of functions in $L^2(0,T\text{; }H^1_0(\Omega))$ that have a distributional derivative 
in $L^2(0,T\text{; }H^{-1}(\Omega))$, i.e. 
	\[
		W(0,T):=\Bigg\{ y\in L^2(0,T;H^1_0(\Omega))\Big\vert  \ \frac{\partial y}{\partial t}\in L^2(0,T;H^{-1}(\Omega)) \Bigg\},
	\]
which is endowed with the norm
	\[
		\|   y\|   _{W(0,T)}:=\|    y\|   _{L^2(0,T;H^1_0(\Omega))}+\|    \partial y/\partial t \|   _{L^2(0,T;H^{-1}(\Omega))}.
	\]
The Banach space $C([0,T]\text{; }L^2(\Omega))$ consists of all continuous functions $y:[0,T]\to L^2(\Omega)$ 
and is equipped with the norm $\max_{t\in [0,T]}\|y(t)\| _{L^2(\Omega)}$. It is well known that $W(0,T)$ is continuously 
embedded in $C([0,T]\text{; }L^2(\Omega))$ and compactly embedded in $L^2(Q)$. 
For proofs and further details regarding spaces involving time, see \cite{Chipot,Evans,Showalter,V03,Wloka}.\\


The following assumptions, close to those in  \cite{CDJ2022,CM2020,CM2021,CMR,CT16,CT22,CWW,CWW2018},
are standing in all the paper, together with the inequality 
\be \label{Ern}
        r> \max \Big\{2,1+\frac{n}{2} \Big\}      
\ee
for the real number $r$ that appears in some assumptions and many statements below (we also remind that $n \in\{1,2,3\}$). 

\begin{Assumption} \label{exist.1}
The operator $\mathcal A :H^1_0(\Omega)\to H^{-1}(\Omega)$, is given by
		\[
			\mathcal A =-\sum_{i,j=1}^{n}\partial_{x_j}(a_{i,j}(x)\partial_{x_i} y),
		\]
where $a_{i,j}\in L^{\infty}(\Omega)$ satisfy the uniform ellipticity condition
		\[
			\exists \lambda_{\mathcal{A}}>0:\ \lambda_{\mathcal{A}}\vert\xi\vert^2\leq \sum_{i,j=1}^{n} a_{i,j}(x)\xi_{i}\xi_{j} \;\;
               \forall \xi\in \mathbb{R}^n \;\; 	 \textrm{and a.a.}\ x\in \Omega.
		\]
The matrix with components $a_{i,j}$ is denoted by $A$. 
\end{Assumption}     

The functions $f, \, L_0 : Q \times \mathbb R\longrightarrow \mathbb{R}$ of the variables $(x,t,y)$,
and the ``initial'' function $y_0$ have the following properties.

\begin{Assumption}  \label{exist.2}
For every $y \in \mathbb R$, the functions $f(\cdot,\cdot,y) \in L^{r}(Q)$,  $L_0(\cdot,\cdot,y) \in L^{1}(Q)$,  and
$y_0 \in L^{\infty}(\Omega)$. 
For a.e. $(x,t) \in Q$ the first and the second derivatives of $f$ and $L_0$ with respect to $y$ exist and are locally 
bounded and locally Lipschitz continuous, uniformly with respect to $(x,t) \in Q$. Moreover, 
$ \frac{\partial f}{\partial y}(x,t,y) \geq 0$ for a.e. $(x,t) \in Q$ and for all $y\in \mathbb R$.
\end{Assumption}  

\begin{Remark}
The last condition in Assumption \ref{exist.2} can be relaxed in the following way:
	\begin{equation}
		\exists C_{f}\in \mathbb R:\ 
		\frac{\partial f}{\partial y}(x,t,y)\geq  C_f
		\ \textrm{ a.a. }(x,t)\in Q \textrm{ and }\forall y\in \mathbb R,
	\label{lbg}
	\end{equation}
see \cite{Casas-Mateos2020,CMR}. However, this leads to complications in the proofs.
\end{Remark}

\subsection{Facts regarding the linear and the semilinear equation}

Let $0\leq\alpha\in L^{\infty}(Q)$ and $u\in L^2(Q)$. We first consider solutions of the following linear variational equality 
for $y \in W(0,T)$ with $y(\cdot,0)=0$:
\begin{align}   \label{lin.1}
		\int_{0}^{T}\llangle \frac{\partial y}{\partial t},\psi \rrangle\,dt+\int_{0}^{T}\langle \mathcal A y, \psi\rangle\,dt
     =\int_{0}^{T}\langle u,\psi\rangle\,dt-\int_{0}^{T}  \langle \alpha y,\psi\rangle\,dt
\end{align}
for all $\psi\in L^2(0,T, H^1_0(\Omega))$, that is, for weak solutions of the equation \eqref{see1} with $f(x,t,y):=\alpha(x,t) y$ and $y_0=0$.

\begin{Theorem}
Let $0\leq \alpha\in L^\infty(Q)$ be given.
\begin{enumerate}
	\item For each $u\in L^2(Q)$ the linear parabolic equation \eqref{lin.1} has a unique weak solution $y_{u}\in W(0,T)$. 
Moreover, there exists a constant $\hat C>0$  independent of $u$ and $\alpha$ such that
	\begin{equation}
		\|y_u\| _{L^2(0,T,H^1_0(\Omega))}\leq \hat C \|u\| _{L^{2}(Q)}.
	\label{wl2}
	\end{equation}
	\label{uniqueexlin}
\item If, additionally, $u\in L^{r}(Q)$ (we remind \reff{Ern}) then the weak solution $y_u$ of \eqref{lin.1} 
belongs to $W(0,T)\cap C(\bar{Q})$. Moreover, there exists a constant $C_r>0$ independent of $u$ and  $\alpha$ such that
	\begin{equation}
		\|y_u\|_{L^2(0,T,H^1_0(\Omega))}+\|y_u\|_{C(\bar Q)}\leq C_r \|u\| _{L^{r}(Q)}.
	\label{clr}
	\end{equation}
	\label{C-est}
\end{enumerate}
\label{mainex}
\end{Theorem}

Besides the independence of the constants $\hat C,$ and $C_r$ on $\alpha$ all claims of the theorem are well known, 
see \cite[Theorem 3.13, Theorem 5.5]{Troltzsch2010}. A proof of a similar independence statement can be found 
in \cite{CDJ2022} for a linear elliptic PDE of non-monotone type. 

\begin{proof}  For convenience of the reader, we prove that the estimates 
are independent of $\alpha$. This is done along the lines of the proof of \cite[Lemma 2.2]{CDJ2022}.
By $y_{0,u}$ we denote a solution of \eqref{lin.1} for $\alpha\equiv0$. It is well known that in this case 
there exist constants $C_r, \hat C>0$ such that
	\begin{equation*}
		\|y_{0,u}\|_{C(\bar Q)}\leq C_r \|u \|_{L^r(Q)},\ \ \|y_{0,u} \|_{L^2(Q)}\leq \hat C \|u\|_{L^2(Q)}.
	\end{equation*}
To apply this, we decompose $u$ in positive and negative parts, $u=u^+-u^-$, $u^+, u^-\ge 0$. By the weak maximum
 principle \cite[Theorem 11.9]{Chipot}, it follows that $y_{\alpha,u^+}, y_{\alpha,u^-}\ge 0$.
Again by the weak maximum principle, the equation
\begin{equation*}
		\frac{\partial }{\partial t}(y_{\alpha,u^+}-y_{0,u^+})+\mathcal A(y_{\alpha,u^+}-y_{0,u^+}) +\alpha (y_{\alpha,u^+}-y_{0,u^+})
       =-\alpha   y_{0,u^+}
\end{equation*}
implies $0\leq y_{\alpha,u^+}\le y_{0,u^+},
	\label{maxp}$
thus $\|     y_{\alpha,u^+}\|   _{C(\bar Q)} \le \|    y_{0,u^+}\|   _{C(\bar Q)}$. By the same reasoning, it follows that 
$0\leq y_{\alpha,u^-}\le y_{0,u^-}$ and $\|y_{\alpha,u^-}\|   _{C(\bar Q)} \le \|    y_{0,u^-}\|   _{C(\bar Q)}$. Hence,
	\begin{equation*}
	\begin{aligned}
		\|     y_{\alpha,u}\|   _{C(\bar Q)}&\le   \|    y_{\alpha,u^+}\|   _{C(\bar Q)} +  \|    
       y_{\alpha,u^-}\|   _{C(\bar Q)}\le \|    y_{0,u^+}\|   _{C(\bar Q)} +  \|
		 y_{0,u^-}\|   _{C(\bar Q)}\\
		&\le C_r(\|   u^+ \|   _{L^r(Q)} +\|    u^-\|   _{L^r(Q)} )\le 2C_r \|    u\|   _{L^r(Q)}.
	\end{aligned}
	\end{equation*}
The estimate for $L^2(0,T,H^1_0(\Omega))$ can be obtained by similar arguments as in \cite{CDJ2022}.
\end{proof}

The next lemma is motivated by an analogous result for linear elliptic equations \cite[Lemma 2.3]{CDJ2022}, 
although, according to the nature of the parabolic setting, the interval of feasible numbers $s$, is smaller.

\begin{Lemma} \label{estLs}
Let $u\in L^{r}(Q)$ and $0\leq \alpha\in L^\infty(Q)$. Let $y_u$ be the unique solution of \eqref{lin.1} and 
let $p_u$ be a solution of the problem
\begin{align}
		\left\{\begin{array}{l}
		-\frac{\partial p}{\partial t}+\mathcal A^*p+\alpha p = u\  \text{ in }\ Q,\\
		p=0 \text{ on } \Sigma,\ p(\cdot,T)=0\ \text{ on } \Omega.
	\end{array} \right.
	\label{adjlin}
\end{align}
Then, for any $s_n\in [1,\frac{n+2}{n})$ there exists a constant $C_{s'_{n}}>0$ independent of $u$ and $\alpha$ such that 
	\begin{equation}
		\max\{\|y_u\|_{L^{s_n}(Q)},\| p_u\|   _{L^{s_n}(Q)}\} \leq C_{s'_{n}}\|u\|_{L^{1}(Q)}.
	\label{aeqestLs}	
	\end{equation}
Here $s'_n$ denotes the H\"older conjugate of $s_n$.
\end{Lemma}

\begin{proof}
First we observe that by Theorem \ref{mainex}, $y_u \in C(\bar Q)\cap W(0,T)$ and as a consequence, 
$\vert  y_u\vert  ^{s_n - 1}\text{sign}(y_u) \in L^{s'_n}(Q)$. Moreover, $s_n < \frac{n+2}{n}$ implies that 
$s'_n >1+ \frac{n}{2}$. By change of variables, see for instance \cite[Lemma 3.17]{Troltzsch2010}, 
a solution of equation \eqref{adjlin}
transforms into a solutions of \eqref{lin.1}. Thus according to Theorem \ref{mainex}, the solution $q$ of
	\begin{align*}
		\left\{\begin{array}{l}
		-\frac{\partial q}{\partial t}+\mathcal A^*q+\alpha q= \vert  y_u\vert  ^{s_n - 1}\text{sign}(y_u)\  \text{ in }\ Q,
		\\  q=0 \text{ on } \Sigma,\ q(\cdot,T)=0\ \text{ on } \Omega.
	\end{array} \right.
	\end{align*}
 belongs to $W(0,T)\cap C(\bar Q)$ and satisfies
	 \[
		 \|   q\|   _{C(\bar Q)} \le C_{s'_n}\|   \vert  y_u\vert  ^{s_n - 1}\text{sign}(y_u)\|   _{L^{s'_n}(Q)} = 
             C_{s'_n}\|   y_u\|   ^{s_n - 1}_{L^{s_n}(Q)},
	 \]
 where $C_{s'_n}$ is independent of $a$ and $v$. Using these facts we derive the equalities
\begin{align*}
	\|   y_u\|   ^{s_n}_{L^{s_n}(Q)} &= \int_{Q}\vert  y_u\vert  ^{s_n}\dx=\big\langle -\frac{\partial q}{\partial t}+ 
    \mathcal A^*q+\alpha q,y_u \big\rangle = \big\langle \frac{\partial  y_u}{\partial t} +\mathcal Ay_u+\alpha y_u, q \big\rangle\\
	& = \int_{Q}uq\dx \le \|u\|   _{L^1(Q)}\|   q\|   _{C(\bar Q)} \le C_{s'_n}\|u\|   _{L^1(Q)}\|   y_u\|   ^{s_n - 1}_{L^{s_n}(Q)}.
\end{align*}
This proves \eqref{aeqestLs} for $y_u$. 
To obtain \eqref{aeqestLs} for $p_u$, one tests \eqref{adjlin} with a weak solution of
	\begin{align*}
		\left\{\begin{array}{l}
		\frac{\partial y}{\partial t}+\mathcal Ay+\alpha y= \vert  q_u\vert  ^{s_n - 1}\text{sign}(q_u)\  \text{ in }\ Q,
		\\  y=0 \text{ on } \Sigma,\ y(\cdot,0)=0\ \text{ on } \Omega,
		\end{array} \right.
	\end{align*}
and argues in an analogous way.
\end{proof}

Below we remind several results for the semilinear equation \eqref{see1}, which will be used further.
A proof of the next theorem can be found in \cite[Theorem 2.1]{CM2020} or \cite[Theorem 2.1]{Troltzsch2010}.

\begin{Theorem}
For any $u\in L^2(Q)$ the semilinear parabolic initial-boundary value problem \eqref{see1} has a unique weak solution 
$y_u\in W(0,T)$. If $u\in L^r(Q)$ (see \reff{Ern}) then $y_u\in W(0,T)\cap L^\infty(Q)$. 
If additionally $y_0\in C(\bar \Omega)$, 
then $y_u\in C(\bar Q)$. Moreover, there exists a constant $D_{r}>0$, independent of $u,f,y_0$ such that 
\begin{equation}
	\|y_u \|    _{W(0,T)}+\|y_u \|   _{L^\infty(Q)} \leq D_{r} \big(\| u \|   _{L^r(Q)}+\|   f(\cdot,\cdot,0) \| _{L^r(Q)}
          +\|    y_0\|   _{L^\infty(\Omega \big)}).	
	\label{semilin}
\end{equation}
Finally, if $u_k\rightharpoonup u$ weakly in $L^{r}(Q)$, then 
\begin{equation}
	\|y_{u_k}-y_u\|_{L^{\infty}(Q)}+\|y_{u_k}-y_u\|_{L^2(0,T;H^1_0(\Omega))}\to 0.
\label{semilinweak}
\end{equation}
\label{estsemeq}
\end{Theorem}

The differentiability of the control-to-state operator under the assumptions \ref{exist.1} and \ref{exist.2} is well known, 
see among others \cite[Theorem 2.4]{CMR}.

\begin{Theorem}
The control-to-state operator $\mathcal G: L^r(Q)\to W(0,T)\cap L^{\infty}(Q)$, defined as $\mathcal G(v):=y_v$, 
is of class $C^2$ and for every $u,v,w\in L^r(Q)$, it holds that 
$z_{u,v}:=\mathcal \mathcal \mathcal G'(u)v$ is the solution of 
\begin{align}
	\left\{\begin{array}{l}
		\frac{d z}{dt}+\mathcal Az+f_y(x,t,y_u) z =v\  \text{ in }\ Q,
		\\ z=0 \ \text{ on }\ \Sigma,\ z(\cdot,0)=0\ \text{ on } \Omega
	\end{array} \right.
	\label{stddt}
\end{align}
and $\omega_{u,(v,w)}:=\mathcal \mathcal \mathcal G''(u)(v,w)$ is the solution of
\begin{align}
	\left\{\begin{array}{l}
	\frac{d z}{dt}+\mathcal Az+f_y(x,t,y_u) z = -f_{yy}(x,t,y_u) z_{u,v } z_{u,w} \  \text{ in }\ Q,
	\\ z=0 \ \text{ on }\ \Sigma,\ z(\cdot,0)=0\ \text{ on } \Omega.
	\end{array} \right.
	\label{stdddt}
\end{align}
\label{derivative}
\end{Theorem}

In the case $v=w$, we will just write $\omega_{u,v}$ instead of $\omega_{u,(v,v)}$.

\begin{Remark}  \label{rebound}
By the boundedness of $\mathcal U$ in $L^\infty(Q)$ and by Theorem \ref{estsemeq}, there exists a constant 
$M_\mathcal U > 0$ such that
	\begin{equation}
	\max\{ \| u\|_{L^\infty(Q)},\|y_u\| _{C(\bar{Q})}\}\leq M_{\mathcal U} \quad \forall u\in \mathcal U.
	\label{unibound}
	\end{equation}
	\end{Remark}

\subsection{Estimates associated with differentiability}

We employ results of the last subsection to derive estimates for the state equation \eqref{see1} and its linearisation 
\eqref{stddt}. 
These estimates constitute a key ingredient to derive stability results in the later sections. 
The next lemma extends \cite[Lemma 2.7]{CDJ2022} from elliptic equations to parabolic ones. 

\begin{Lemma} \label{sprep}
The following statements are fulfilled.
\begin{itemize}
\item[(i)] There exists a positive constant $M_2$ such that  for every $ u,\bar u \in \Uad \text{ and }v\in L^r(Q)$
\begin{align}
  \|   z_{u,v} - z_{\bar u,v}\|   _{L^2(Q)}\le M_2\|   y_u - y_{\bar u}\|   _{C(\bar Q)}\|   z_{\bar u,v}\|   _{L^2(Q)}.\label{E33}
\end{align}
\item[(ii)] Let $X=C(\bar Q)$ or $X=L^2(Q)$. Then there exists $\varepsilon > 0$ such that for every 
$u, \bar u \in \Uad$ with $\|   y_u - y_{\bar u}\|_{C(\bar Q)} < \varepsilon$ the following inequalities are satisfied
\begin{align}
	&\|   y_u - y_{\bar u}\|_X \le 2\|   z_{\bar u,u - \bar u}\|   _X \le 3\|   y_u - y_{\bar u}\|_X,
	\label{E2.14.2}\\
	&\|   z_{\bar u,v}\|   _X \le 2 \|   z_{u,v}\|   _X \le 3\|   z_{\bar u,v}\|   _X.\label{E2.15.2}
\end{align}
\end{itemize}
\end{Lemma}

The proof, that is a consequence of Lemma \ref{spreprep}, is given in Appendix A.


\section{The control problem} \label{SProblem}
The optimal control problem \eqref{ocp1}-\eqref{const} is well posed under assumptions \ref{exist.1} and \ref{exist.2}. 
Using the direct method of calculus of variations one can easily prove that there exists at least one global minimizer, 
see \cite[Theorem 5.7]{Troltzsch2010}. On the other hand, the semilinear state equation makes the optimal 
control problem nonconvex, therefore we allow global minimizers as well as local ones. In the literature, weak and 
strong local minimizers are considered. 

\begin{Definition}
We say that $\bar u \in \mathcal U$ is an $L^r (Q)$-weak local minimum of problem  \eqref{ocp1}-\eqref{const}, if there exists 
some $\varepsilon>0$ such that
\[
J( \bar u) \le J(u) \  \ \ \forall u \in \mathcal U \text{ with } \|   u-\bar u\|   _{L^r(Q)}\le \varepsilon.
\]
We say that $\bar u \in \mathcal U$ a strong local minimum of \Pb if there exists $\varepsilon >0$ such that
\[
J( \bar u) \le J(u) \ \ \ \forall u \in \mathcal U \text{ with } \|    y_u-y_{\bar u}\|   _{L^{\infty}(Q)}\le \varepsilon.
\]
We say that $\bar u \in \mathcal U$ is a strict (weak or strong) local minimum if the above inequalities are
strict for $u\ne \bar u $.
\end{Definition}

Relations between these types of optimality are obtained in \cite[Lemma 2.8]{Casas-Mateos2020}.


As a consequence of Theorem \ref{derivative} and the chain rule, we obtain the differentiability of the objective functional
with respect to the control.
 
\begin{Theorem} \label{T3.1}
The functional $J:L^r(Q) \longrightarrow \mathbb{R}$ is of class $C^2$. Moreover, 
given $u, v, v_1, v_2 \in L^r(Q)$ we have
\begin{align}
     J'(u)v& =\int_Q\Big(\frac{dL_0}{dy}(x,t,y_u)+mu\Big)z_{u,v}+(my_u+g)v\dx\dt\\
      &= \int_Q(p_u+my_u+g)v\dx\dt,\label{E3.4}\\
   J''(u)(v_1,v_2)& = \int_Q\Big[\frac{\partial^2L}{\partial y^2}(x,t,y_u,u) - p_u\frac{\partial^2f}{\partial y^2}(x,t,y_u)\Big]
       z_{u,v_1}z_{u,v_2}\dx\dt\label{E3.5.a}\\
        &+ \int_Qm(z_{u,v_1}v_2+z_{u,v_2}v_1)\dx\dt,\label{E3.5}
\end{align}
Here, $p_u \in W(0,T) \cap C(\bar Q)$ is the unique solution of the adjoint equation
\begin{equation}
      \left\{\begin{array}{l}\displaystyle -\frac{d p}{dt}+\mathcal{A}^*p + \frac{\partial f}{\partial y}(x,t,y_u)p=  
    \frac{\partial L}{\partial y}(x,t,y_u,u) \text{ in } Q,\\ p = 0\text{ on } \Sigma, \ p(\cdot,T)=0\text{ on } \Omega.\end{array}\right.
\label{E3.6}
\end{equation}
\end{Theorem}

We introduce the Hamiltonian 
$Q\times\mathbb R\times\mathbb R\times \mathbb R \ni (x,t,y,p,u) \mapsto H(x,t,y,p,u) \in \mathbb R$
in the usual way:
\begin{align*}
       H(x,t,y,p,u):=L(x,t,y,u)+p(u-f(x,t,y)).
\end{align*}
The local form of the Pontryagin type necessary optimality conditions for problem \eqref{ocp1}-\eqref{const} in the next theorem
is well known (see e.g. \cite{Casas-Mateos2020,CMR,Troltzsch2010}).
 
\begin{Theorem}
If $\bar u$ is a weak local minimizer for problem  \eqref{ocp1}-\eqref{const}, 
then there exist unique elements $\bar y, \bar p\in W(0,T)\cap C(\bar Q)$ such that 
\begin{align}
& \left\{\begin{array}{l} \frac{d \bar y}{dt}+\mathcal{A}\bar y + f(x,t,\bar y) = \bar u \text{ in } Q,\\ \bar y = 
0\text{ on } \Sigma, \ \bar y(\cdot,0)=y_0\text{ on } \Omega.\end{array}\right.
\label{E3.7}\\
&\left\{\begin{array}{l}\displaystyle \frac{d \bar p}{dt}+\mathcal{A}^*\bar p =
 \frac {\partial H}{\partial y}(x,t,\bar y,\bar p,\bar u) \text{ in } Q,\\ \bar p = 0\text{ on } \Sigma, \ \bar p(\cdot,T)=0\text{ on } \Omega.\end{array}\right.
\label{E3.8}\\
&\int_Q\frac {\partial H}{\partial u}(x,t,\bar y,\bar p,\bar u)(u - \bar u)\dx\dt \ge 0 \quad \forall u \in \Uad. \label{E3.9}
\end{align}
\label{pontryagin}
\end{Theorem}


\subsection{Sufficient conditions for optimality and stability} \label{CS}

In this subsection we discuss the state of the art in the theory of  sufficient second order 
optimality conditions in PDE optimal control, as well as related stability results for the optimal solution.
For this purpose, we recall the definitions of several cones that are useful in the study of sufficient 
conditions. Given a triplet $(\bar y, \bar p, \bar u)$ satisfying the optimality system in Theorem~\ref{pontryagin}, 
and abbreviating $\frac {\partial \bar H}{\partial u}(x,t):=
\frac {\partial H}{\partial u}(x,t,\bar y,\bar p,\bar u)$, we have from \eqref{E3.9} that almost everywhere in $Q$
\[
     \bar u=u_a\; \text{ if }\;   \frac{\partial \bar H}{\partial u}> 0\  \ \  \text{ and } \ \ \bar u=
        u_b \; \text{ if } \;  \frac{\partial \bar H}{\partial u}< 0.
\]
This motivates to consider the following set
\begin{align}
\Big\{v\in L^2(Q)\Big \vert   v\geq 0\text{ a.e. on } [\bar u =u_a]\text{ and } v\leq 0 \text{ a.e. on } [\bar u =u_b]\Big\}.
\label{sign}
\end{align}
Sufficient second order conditions for (local) optimality based on \eqref{sign} are given in \cite{CMR,Casas-Mateos2020, CT16}.
Following the usual approach in mathematical programming, one can define the critical cone at $\bar u$ as follows:
\begin{equation*}
     C_{\bar u}:=\Big\{v\in L^2(Q)\Big \vert   v\text{ satisfies }\eqref{sign}\text{ and } v(x,t)=
      0\text{ if }\Big\vert  \frac{\partial \bar H}{\partial u}(x,t)\Big \vert   >0\Big\}.
\end{equation*}
Obviously, this cone is trivial if $\frac{\partial \bar H}{\partial u}(x,t) \not= 0$ for a.e. $(x,t)$ 
(which implies bang-bang structure of $\bar u$) thus no additional information 
can be gained based on $C_{\bar u}$.  
To address this issue, it was proposed in \cite{Dunn1998,MZ1979} to consider larger cones 
on which second order conditions can be posed. 
Namely, for $\tau>0$ one defines
\begin{align}
D^{\tau}_{\bar u}&:=
\Big\{v\in L^2(Q)\Big \vert   v\text{ satisfies }\eqref{sign}\text{ and } v(x,t)=
0\text{ if }\Big\vert  \frac{\partial \bar H}{\partial u}(x,t)\Big \vert   >\tau\Big\},\\
G^{\tau}_{\bar u}&:=
\Big\{v\in L^2(Q)\Big \vert v\text{ satisfies }\eqref{sign}\text{ and } J'(\bar u)(v)\leq \tau \|z_{\bar u,v} \|_{L^1(Q)}\Big\},\\
E^{\tau}_{\bar u}&:=
\Big\{v\in L^2(Q)\Big \vert   v\text{ satisfies }\eqref{sign}\text{ and } J'(\bar u)(v)\leq \tau \|z_{\bar u,v} \|_{L^2(Q)}\Big\},\\
C^\tau_{\bar u}&:=D^{\tau}_{\bar u}\cap G^{\tau}_{\bar u}.
\end{align}
The cones $D^\tau_{\bar u}$, $E^{\tau}_{\bar u} \text{ and } G^{\tau}_{\bar u}$ were introduced in \cite{Casas12,CT16}
as extensions of the usual critical cone. It was proven in \cite{Casas12,CRT15,CT16} that the condition:
\begin{equation}
    \exists \delta>0, \tau>0 \ \ \mbox{ such that } \ \ J''(\bar u)v^2\ge \delta \|z_{\bar u,v} \|_{L^2(Q)}^2 \ \  \forall v \in G
\label{soclin}
\end{equation}
is sufficient for weak (in the case $G=D^{\tau}_{\bar u}$) or strong (in the case 
$G=E^{\tau}_{\bar u}$) local optimality in the elliptic and parabolic setting. Most recently, 
the cone $C^\tau_{\bar u}$ was defined in \cite{Casas-Mateos2020} and also used in \cite{CM2021}.
It was proved in \cite{Casas-Mateos2020}, that \eqref{soclin} with $C=C^\tau_{\bar u}$ is sufficient for 
strong local optimality.\\
Under \eqref{soclin} it is possible to obtain some stability results. In \cite{CRT15} and \cite{CT16} the authors obtain Lipschitz 
stability in the ($L^2-L^\infty$)-sense for the states%
\footnote{ \label{Fn-} For $p,r\in [1,\infty]$, we speak of stability in the $L^p-L^r$-sense for the optimal states 
$\bar y$ with respect to perturbations (may appear in the equation or the objective) $\xi$, if there exists a constant 
$\kappa>0$ such that $\| y^{\xi}-\bar y\|_{L^p(Q)}\leq \kappa \|\xi \|_{L^r(Q)}$, for all $\xi$ that are sufficiently small. 
Here, $y^\xi$ denotes the state corresponding to the perturbation $\xi$.
We use this expression analogously for the optimal controls.}, under perturbations appearing in a tracking type 
objective functional 
and under the assumption that the perturbations are Lipschitz. Further they obtain H\"older stability for 
the states under a Tikhonov type perturbation. H\"older stability under \eqref{soclin} with exponent $1/2$ was proved 
in \cite{CT22} with respect to perturbations in the initial condition.

To improve the stability results an additional assumption is needed. This role is usually played by the structural 
assumption on the adjoint state or generally on the derivative of the Hamiltonian with respect to the control. 
In the case of an elliptic state equation, 
\cite{Qui-Wachsmuth2018} uses the structural assumption
\begin{equation}
        \exists \kappa > 0 \text{ such that }  \ \ \Big\vert\Big\{x \in \Omega : \Big\vert  \frac {\partial \bar H}{\partial u} \Big\vert   
       \le \varepsilon\Big\}\Big\vert   
     \le \kappa\varepsilon\quad \forall \varepsilon > 0.
\label{struct}
\end{equation}
In the parabolic case this assumption (with  $\Omega$ replaced with $Q$) is used in \cite{CT22}. 
We recall that the assumption \eqref{struct} implies that $\bar u$ is of bang-bang type. 
Further, \eqref{struct} implies the existence of a constant $\tilde \kappa>0$ such that the following growth property holds:
\begin{equation}
J'(\bar u)(u-\bar  u)\geq \tilde \kappa \|u-\bar u\|_{L^1(X)}^2 \ \forall u\in \mathcal U.
\label{growthfirst}
\end{equation}
For a proof see \cite{ASS16}, \cite{OSVE2} or \cite{S2015}. If the control constraints satisfy $u_a<u_b$ almost everywhere 
on $Q$, both conditions, \eqref{struct} and \eqref{growthfirst} are equivalent, 
see \cite[Proposition 6.4]{DJV2022}. In \cite{Qui-Wachsmuth2018}, using \eqref{struct} and \eqref{soclin} 
with $G=D^\tau_{\bar u}$,
the authors proof $L^1$-Lipschitz stability of the controls for an elliptic semilinear optimal control problem 
under perturbations appearing simultaneously in the objective functional and the state equation.
Assuming \eqref{struct}, \eqref{soclin} may also be weakened to the case of negative curvature,
\begin{equation} \label{soclinn}
      \exists \delta < \tilde \kappa, \ \exists \tau>0 \ \  \mbox{such that} \ \ J''(\bar u)v^2\ge -\delta \| v \|_{L^1(\Omega)}^2 
           \ \   \forall v\in C^{\tau}_{\bar u}.
\end{equation}
In \cite{CWW}, \cite{CWW2018} it was proved that \eqref{struct} together with \eqref{soclinn} 
implies, for the semililnear elliptic case, weak local optimality in $L^1(\Omega)$. Lipschitz stability results 
were also obtained in \cite{DJV2022} in the elliptic case. Finally,  for a semilinear parabolic equation
with perturbed initial data, 
\cite[Theorem 4.6]{CT22} obtains, under 
\eqref{soclin} and \eqref{struct}, $L^2-L^2$ and $L^1-L^2$-H\"older stability (see Footnote \ref{Fn-}), 
with exponent $\frac{2}{3}$, 
for the optimal states and controls respectively. Additionally, Lipschitz dependence is obtained on perturbations 
in $L^\infty(Q)$.

\section{A unified sufficiency condition} \label{SApproach} \label{SUni}

In this section, we introduce an assumption that unifies the first and second order conditions presented in the previous section.

\begin{Assumption} \label{A4}
For a  number $k\in \{0,1,2\}$, at least one of the following conditions is fulfilled: 

\smallskip\noindent
($A_k$): There exist constants $\alpha_k,\gamma_k>0$ such that
\begin{equation}  \label{E3.13.1} 
   J'(\bar u)(u - \bar u) + J''(\bar u)(u - \bar u)^2 \ge \gamma_k\|   z_{\bar u,u - \bar u}\|^{k}_{L^2(Q)}\|   u -\bar u\|^{2-k}_{L^1(Q)}
\end{equation}
for all $u \in \Uad \text{ with } \|   y_u - \bar y\|   _{C(\bar Q)} < \alpha_k$.

\smallskip\noindent
($B_k$): There exist constants $\tilde \alpha_k, \tilde \gamma_k>0$ such that \reff{E3.13.1} holds 
	for all $u \in \Uad$ such that $\|   u - \bar u\|_{L^1(Q)} < \tilde\alpha_k$.
\end{Assumption}

In the context of optimal control of PDE's the assumptions ($A_0$) and ($B_0$) were first introduced  
in \cite{DJV2022} and for $k=1,2$ in \cite{CDJ2022}.  Assumption \ref{A4}($B_0$) originates from optimal control theory 
of ODE's where it was first introduced in \cite{OSVE} to deal with nonlinear affine optimal control problems. 
The cases $k=1,2$ are extensions, adapted to the nature of the PDE setting, while the case $k=0$
can be hard to verify if a structural assumption like \reff{struct} is not imposed. The assumptions corresponding to $k=1,2$ 
are applicable for the case of optimal controls that need not be bang-bang, especially the case $k=2$ seems 
natural for obtaining state stability. Assumption ($A_k$) implies strong (local) optimality, while
Assumption ($B_k$)  leads to weak (local) optimality. As seen below, in some cases the two assumptions are equivalent. 

For an optimal control problem subject to an semilinear elliptic equation the claim of the next proposition with
$k=0$ was proven in \cite[Proposition 5.2]{CDJ2022}. 

\begin{Proposition} \label{equivasu}
For any $k\in\{0,1,2\}$, Assumption ($A_k$) implies ($B_k$).
If $\bar u $ is bang-bang 
(that is, $\bar u(x,t) \in \{ u_a(x,t), u_b(x,t) \}$ for a.e. $(x,t) \in Q$)  then assumptions ($A_k$) and ($B_k$) are equivalent.
\end{Proposition}

The proof is given in Appendix A.

\begin{Remark}
We compare the items in Assumption \ref{A4} to the ones using \eqref{struct} and \eqref{soclinn} or \eqref{soclin}.
\begin{enumerate}
\item Assumption \ref{A4}($A_0$)  is implied by the structural assumption \eqref{struct} and also allows for negative 
curvature, similar to \eqref{soclinn}. For details see \cite[Theorem~6.3]{DJV2022}.
\item Assumption \ref{A4}($A_1$)  is implied by the structural assumption \eqref{struct} together with \eqref{soclin}. 
This is clear 
by \eqref{growthfirst} and by using $v$ and $w$ as defined in Lemma \ref{vandw} and arguing as in Corollary \ref{equi1},
both presented below in this section.
\item Assumption \ref{A4}($A_2$) is implied by \eqref{soclin} together with the first order necessary condition. 
\end{enumerate}
\end{Remark}

\subsection{Sufficiency for optimality of the unified condition} \label{SSsufU}

In this subsection we show that  assumptions \ref{A4}($A_k$) and ($B_k$) are sufficient either for strict weak 
or strict strong local optimality, correspondingly.

\begin{Theorem}
The following holds.
\begin{enumerate}
\item Let $m=0$ in \reff{genobj}. Let $\bar u \in \Uad$ satisfy the optimality conditions \eqref{E3.7}--\eqref{E3.9} 
and Assumption \ref{A4}($A_k$) with some $k\in \{0,1,2\}$. Then, there exist $\varepsilon_k,\kappa_k > 0$ such that:
\begin{equation} \label{E3.12}
      J(\bar u) + \frac{\kappa_k}{2}\|   y_u - \bar y\|   ^{k}_{L^2(Q)}\|   u - \bar u\|   _{L^1(Q)}^{2-k} \le J(u)
\end{equation}
for all $u \in \Uad \text{ such that } \| y_u - \bar y\|   _{C(\bar Q)} < \varepsilon_k$.
\item Let $\bar u \in \Uad$ satisfy the optimality conditions \eqref{E3.7}--\eqref{E3.9} and Assumption \ref{A4}($B_k$) 
with some $k\in \{0,1,2\}$. Then, there exist $\varepsilon_k,\kappa_k > 0$ such that \eqref{E3.12} holds
 for all $u \in \Uad \text{ such that } \|   u - \bar u\|   _{L^1(Q)} < \varepsilon_k$.
\end{enumerate}
\label{T3.3}
\end{Theorem}

Before presenting a proof of Theorem \ref{T3.3}, we establish some technical results. 
The following lemma was proved for various types of objective functionals, see e.g.
\cite[Lemma 6]{CT16},\cite[Lemma 3.11]{CRT15}. Nevertheless, our objective functional is more general, 
therefore we present in Appendix A an adapted proof.

\begin{Lemma} Let $\bar u \in \Uad$. The following holds.
\begin{enumerate}
\item Let $m=0$ hold. For every $\rho > 0$ there exists $\varepsilon > 0$ such that 
\begin{align}
 \vert  [J''(\bar u + \theta(u - \bar u)) - J''(\bar u)](u-\bar u)^2\vert \leq \rho\|z_{\bar u,u-\bar u}\|^2_{L^2(Q)}
\label{eins}
\end{align}
for all $u \in \Uad$ with $\|   y_u - \bar y\|   _{C(\bar Q)} < \varepsilon$ and $\theta \in [0,1]$. 
\item For every $\rho > 0$ there exists $\varepsilon > 0$ such that \eqref{eins} holds
for all $u \in \Uad$ with $\| u - \bar u\|_{L^1( Q)} < \varepsilon$ and $\theta \in [0,1]$. 
\label{zwei}
\end{enumerate}
\label{biglemmat}
\end{Lemma}

For the assumptions with $k\in \{0,1\}$, we need the subsequent corollary, which is also given in Appendix A.

\begin{Corollary} \label{bigcorollary}
 Let $\bar u \in \Uad$. The following holds for $m=0$:
\begin{enumerate}
\item  For every $\rho > 0$ there exists $\varepsilon > 0$ such that
\begin{equation}
\vert  [J''(\bar u + \theta(u - \bar u)) - J''(\bar u)](u-\bar u)^2\vert   \leq \rho\|   z_{\bar u,u-\bar u}\|_{L^2(Q)}\|   u-\bar u\| _{L^1(Q)} 
\label{E3.17.1}
\end{equation}
for all $u \in \mathcal U$ with $\|   y_u - \bar y\|   _{C(\bar Q)} < \varepsilon$  and for all $ \theta \in [0,1]$.
\item For every $\rho > 0$ there exists $\varepsilon > 0$ such that
\begin{equation}
\vert  [J''(\bar u + \theta(u - \bar u)) - J''(\bar u)](u-\bar u)^2\vert   \leq \rho\|   u-\bar u\|_{L^1(Q)}^2 
\label{E3.17.2}
\end{equation}
for all $u \in \Uad$ with $\| y_u - \bar y\|_{C(\bar Q)} < \varepsilon$ and for all $\theta \in [0,1]$.
\end{enumerate}
The same assertions hold for $m\neq 0$ if one requires $\|u-\bar
u\|_{L^1(Q)}$ to be small instead of $\|   y_u - \bar y\| _{C(\bar Q)}$.
\end{Corollary}

The next lemma clams that Assumption \ref{A4} implies a growth similar to \reff{E3.12} of the first derivative 
of the objective functional in a neighborhood of $\bar u$. 

\begin{Lemma}The following claims are fulfilled.
\begin{enumerate}
	\item Let $m=0$ and $\bar u$ satisfy assumption $(A_k)$, for some $k\in \{0,1,2\}$.
	Then, there exist $\bar \alpha_k,\bar\gamma_k > 0$ such that
	\begin{equation}
		J'(u)(u - \bar u) \ge \bar\gamma_k\|   z_{\bar u,u - \bar u}\|   ^{k}_{L^2(Q)}\|u - \bar u\|^{2-k}_{L^1(Q)}
		\label{E4.6}
	\end{equation}
	 for every $u \in \Uad \text{ with } \|   y_u - \bar y\|   _{C(\bar Q)} < \bar \alpha_k$.
	\item Let $\bar u$ satisfy assumption $(B_k)$ for some $k\in \{0,1,2\}$.
	Then, there exist $\bar \alpha_k ,\bar\gamma_k > 0$ such that \reff{E4.6} holds
	for every $u \in \Uad \text{ with }\| u -\bar u\|   _{L^1(Q)} < \bar \alpha_k$.
	\end{enumerate}
	\label{good}
\end{Lemma}

\begin{proof} Since $J$ is of class $C^2$ we can use the mean value theorem to infer the existence of 
a function $\theta :Q\to [0,1]$ such that
\[
J'(u)(u - \bar u)-J'(\bar u)(u - \bar u) =J''(\bar u + \theta(u - \bar u))(u-\bar u)^2
\]
and under $(A_k)$ in Assumption \ref{A4}, we infer the existence of positive constants $\gamma_k$ and $\alpha_k$ such that
	\begin{align*}
		J'(u)(u - \bar u) &=J'(\bar u)(u - \bar u) + J''(\bar u)(u - \bar u)^2+[J'(u)(u - \bar u)-J'(\bar u)(u - \bar u) - J''(\bar u)(u - \bar u)^2]\\
		&\ge \gamma_k\|   z_{\bar u,u - \bar u}\|   ^{k}_{L^2(Q)}\|   u - \bar u\|   ^{2-k}_{L^1(Q)}-\vert  [J''(\bar u + \theta(u - \bar u)) - J''(\bar u)](u-\bar u)^2\vert  ,
	\end{align*}
for all $u\in \mathcal U$ with $\|    y_u-\bar y\|   _{C(\bar Q)}<\alpha_k$.
Using Lemma \ref{biglemmat}, we obtain that
	\begin{align*}
J'(u)(u - \bar u)&\ge (\gamma_k-\rho_k)\|   z_{\bar u,u - \bar u}\|   ^{k}_{L^2(Q)}\|   u - \bar u\|   ^{2-k}_{L^1(Q)}
	\end{align*}
for all $u\in \mathcal U$ with $\|    y_u-\bar y\|   _{C(\bar Q)}<\bar \alpha_k$ and $\bar{\alpha_k}:=
\min\{\alpha_k,\varepsilon_k \}$, 
where $\varepsilon_k>0$ is chosen such that $\bar\gamma_k:=\gamma_k-\rho_k>0$ holds. 
Using Corollary \ref{bigcorollary} and the estimate 
$\|y_u - \bar y\|_{L^\infty(Q)}\le C_r(2M_{\mathcal U})^{\frac{r-1}{r}}\|u - \bar u\|_{L^1(Q)}^{\frac{1}{r}}$, proves 
the case for \eqref{E4.6}.
\end{proof}
Finally, we conclude this subsection with the proof of Theorem \ref{T3.3}.\\

{\em Proof of Theorem \ref{T3.3}.}
Using the Taylor expansion and the optimality condition $J'(\bar u)(u - \bar u) \geq 0$ we have
\begin{align*}
J(u)& = J(\bar u) + J'(\bar u)(u - \bar u) + \frac{1}{2}J''(u_\theta)(u - \bar u)^2 \geq J(\bar u) + \frac{1}{2} J'(\bar u)(u - \bar u) + \frac{1}{2}J''(u_\theta)(u - \bar u)^2
\end{align*}
where $u_\theta:=\bar u+\theta(u-\bar u)$, with $\theta: Q \to [0,1]$. We continue this inequality, using that
by Assumption \ref{A4} there exist $\alpha_k>0$ and $\gamma_k>0$ such that \reff{E3.12} holds:
\begin{align*}
J(u) & \ge J(\bar u) + \frac{1}{2}[J'(\bar u)(u - \bar u) + J''(\bar u)(u - \bar u)^2] + \frac{1}{2}[J''(u_\theta) - J''(\bar u)](u - \bar u)^2]\\
 &\ge J(\bar u) + \frac{\gamma_k}{2} \|   z_{\bar u,u - \bar u}\|   ^k_{L^2(Q)}\|   u - \bar u\|   ^{2-k}_{L^1(Q)} - 
        \frac{1}{2}  \big\vert  [J''(u_\theta) - J''(\bar u)](u - \bar u)^2 \big\vert
\end{align*}
for all $u\in \mathcal U$ with either $\|    y_u-\bar y\|   _{L^\infty(Q)}<\alpha_k$ or $\|   u - \bar u \|   _{L^1(Q)}<\alpha_k$, 
depending on the chosen assumption $(A_k)$ or $(B_k)$.
Now,  either by Lemma \ref{biglemmat} or Corollary~\ref{bigcorollary} (depending on the assumption) there 
exist $\varepsilon>0$ and $\bar \gamma_k<\gamma_k$ such that
 \[
 \vert  [J''(u_\theta) - J''(\bar u)](u - \bar u)^2\vert   \leq
    \bar \gamma_k\|   z_{\bar u,u - \bar u}\|   ^k_{L^2(Q)}\|   u - \bar u\|  ^{2-k}_{L^1(Q)}
 \]
for every $u \in \Uad$ with $\|   y_u - \bar y\|   _{C(\bar Q)} <\varepsilon$. We may choose 
$\bar \alpha_k>0$ and $\bar \gamma_k>0$ according to Lemma~\ref{good} and depending 
on the chosen assumption therein. Inserting this estimate in the above expression and applying \reff{E2.14.2} gives
\begin{align*}
J(u)&\ge J(\bar u) +  \frac{1}{2}(\gamma_k-\bar \gamma_k)\|   z_{\bar u,u - \bar u}\|   ^k_{L^2(Q)}\|   u - \bar u\|   ^{2-k}_{L^1(Q)} \ge J(\bar u) + \frac{3(\gamma_k-\bar \gamma_k)}{4}\|   y_u - \bar y\|   ^{k}_{L^2(Q)}\|   u - \bar u\|   ^{2-k}_{L^1(Q)},
\end{align*}
for all $u\in \mathcal U$  with either 
$\|y_u-\bar u\| _{L^\infty(Q)}<\min\{\varepsilon,\bar\alpha_k\}$ or 
$\|u-\bar u\|_{L^1(Q)}<\min\{\frac{\varepsilon^r}{C_r^r(2M_{\mathcal U})^{(r-1)}},\bar\alpha_k \}$ depending 
on the selected $k\in \{0,1,2\}$. To complete the proof of the second claim of the theorem we use that
\[
\| y_u - \bar y\|   _{L^\infty(Q)}\leq C_r (2M_{\mathcal U})^{\frac{r-1}{r}}\|    u-\bar u\|   _{L^1(Q)}^{\frac{1}{r}}
\]
to apply Lemma \ref{biglemmat} or Corollary \ref{bigcorollary} depending on $k\in \{0,1,2\}$. 
\begin{flushright} $\square$
\end{flushright}
\endproof

\subsection{Some equivalence results for the assumptions on cones}

In this subsection we show that some of the  items in Assumption \ref{A4} can be formulated equivalently 
on the cones $D^\tau_{\bar u}$ or $C^{\tau}_{\bar u}$ respectively. This applies to ($B_k$) or to ($A_k$)
depending on whether the objective functional explicitly depends on the control or not.
We need the next lemma, the proof of which uses a result from \cite{CMCO}.

\begin{Lemma}
Let $\bar u\in \mathcal U$ satisfy the first order optimality condition \eqref{E3.7}-\eqref{E3.9} and let $u\in \mathcal U$ be given. 
For $\tau >0$, we define 
\begin{align*}
v:=\left\{\begin{array}{lll}
 0 &\text{ on } &[ \ \vert \frac {\partial \bar H}{\partial u}\vert  >\tau\ ],\\
 u-\bar u &\text{ else},&
\end{array}\right.
\end{align*}
and $ w:=u-\bar u-v$. Let $\varepsilon>0$ be given. Then there exists a constant $C>0$ such that
\begin{align}
\max\{ \|z_{\bar u,w}\|_{L^{\infty}(Q)}, \ \|z_{\bar u,v}\|   _{L^{\infty}(Q)}\}<C\max\{ \varepsilon,\varepsilon^{\frac{1}{r}}\}
\label{monotone}
\end{align}
 for all $u \in\mathcal U \text{ with }\|  u-\bar u\|_{L^{1}(Q)}<\varepsilon$.
Let $\varepsilon_0>0$ be such that \eqref{E2.14.2} holds. If the control does not appear explicitly 
in \eqref{ocp1} (that is, $m=g=0$ in \reff{genobj}), then \eqref{monotone} holds for all $u\in \mathcal U$ such that 
$u-\bar u \in G^\tau_{\bar u}$ and $\|z_{\bar u,u-\bar u}\|_{L^{\infty}(Q)}<\varepsilon_0$.
\label{vandw}
\end{Lemma}

\begin{proof}
We define $\tilde u, \hat u\in \mathcal U$ by
\begin{align*}
\tilde u:=\left\{\begin{array}{lll}\displaystyle \bar u &\text{ on}&
      [\ \vert  \frac {\partial \bar H}{\partial u}\vert  >\tau\ ],\\ u &\text{ else.}&\end{array}\right. \ \ \hat u:=
     \left\{\begin{array}{lll}\displaystyle  u &\text{ on}& [\ \vert \frac {\partial \bar H}{\partial u}\vert  >
     \tau\ ],\\ \bar u &\text{ else.}\end{array}\right.
\end{align*}
Observe that $v=\tilde u-\bar u$, $w=\hat u-\bar u$ and $u-\bar u =v+w$. It is trivial by construction 
that $\|v\|   _{L^1(Q)},\ \|w\|   _{L^1(Q)}\le \|u-\bar u\|   _{L^1(Q)}$. 
On the other hand, by \eqref{E2.14.2}, $\|z_{\bar u,u-\bar u}\|_{L^{\infty}(Q)}<\varepsilon$ implies  
$\|y_{u}-y_{\bar u}\|_{L^{\infty}(Q)}<2\varepsilon$. If $m,g=0$, we can argue as in \cite{CMCO} 
using $u-\bar u \in G^\tau_{\bar u}$ and the definition of $w$, to estimate
\begin{align*}
\tau \|w\|   _{L^1(Q)}&\le J'(\bar u)(u-\bar u)\le
\tau \|z_{\bar u,u-\bar u} \|_{L^{1}(Q)}.
\end{align*}
Thus by Theorem \ref{mainex} and \eqref{unibound}
\begin{align*}
\|z_{\bar u,w} \|   _{L^\infty(Q)}\le
\left\{\begin{array}{lll}
C_0\|z_{\bar u,u-\bar u}\|_{L^\infty(Q)}^{1/r} &\text{ if }&m,g=0,\ u-\bar u \in G^\tau_{\bar u},\\
C_0\|u-\bar u \|_{L^1(Q)}^{1/r} &\text{ else,}&
\end{array}\right.
\end{align*}
with $C_0:= C_r (2M_{\mathcal U})^{\frac{r-1}{r}}$.
For $z_{\bar u, v}$, we estimate with $C:=2(C_0+1)$
\begin{align*}
    \|z_{\bar u,v} \|_{L^\infty(Q)}&\le \|z_{\bar u,v+w} \|_{L^\infty(Q)}+\|-z_{\bar u,w} \|_{L^\infty(Q)}
     \le C\max\{\varepsilon,\varepsilon^{\frac{1}{r}}\}.
\end{align*}
In the second case the estimate holds trivially.
\end{proof}

Now we continue with the equivalence properties.

\begin{Corollary}
For $k\in \{0,2\}$,  Assumption \ref{A4}$(B_k)$ is equivalent to the following condition ($\bar B_{k}$):
there exist constants $\alpha_k,\gamma_k,\tau>0$ such that 
\begin{equation} \label{EbBk}
        J'(\bar u)(u - \bar u) + J''(\bar u)(u - \bar u)^2 \ge 
            \gamma_k\|   z_{\bar u,u - \bar u}\|^{k}_{L^2(Q)}\| u - \bar u\|^{2-k}_{L^1(Q)},
\end{equation}
for all $u \in \Uad$ for which $(u-\bar u )\in  D_{\bar u}^{\tau}\text{ and } \|u-\bar u \| _{L^{1}(Q)} < \alpha_k$.
\label{equi1}
\end{Corollary}

\begin{proof}
Let $k\in \{0,2\}$. If ($B_k$) holds then ($\bar B_{k}$) is obviously also fulfilled.
Now let ($\bar B_{k}$) hold.  
The numbers $\tilde  \alpha_k$ and $\tilde\gamma_k>0$ will be chosen later so that  assumption $(B_k)$ 
will hold with these numbers. 
For now we only require that $0 < \tilde \alpha_k < \alpha_k$.
Choose an arbitrary $u \in \Uad$ with $\|u-\bar u \| _{L^{1}(Q)} < \tilde \alpha_k$.
We only need to prove \reff{E3.13.1} in the case $u-\bar u\notin D^{\tau}_{\bar u}$. 
Take $v$ and $w$ as defined in Lemma \ref{vandw}. Clearly by definition $v\in D^\tau_{\bar u}$.
As a direct consequence of \eqref{E3.5.a}-\eqref{E3.5} and Assumption \ref{exist.1} and \ref{exist.2}
there exists a constant $C_0>0$ such that 
\begin{align}
\vert  J''(\bar u)(w)^2\vert  &\leq C_0 \|z_{\bar u,w}\|   _{L^\infty(Q)} \|w\|   _{L^1(Q)},\\
\vert  J''(\bar u)(w,v)\vert  &\leq C_0 \|z_{\bar u, v}\|_{L^\infty(Q)} \|w\|   _{L^1(Q)}.
\label{secvest}
\end{align}
We estimate 
\begin{equation} \label{secest}
      \Big\vert J''(\bar u)(w)^2+2J''(\bar u)(w,v) \Big \vert 
      \leq 3C_0(\|z_{\bar u, w}\| _{L^\infty(Q)}+\|z_{\bar u, v}\| _{L^\infty(Q)})\| w \|_{L^1(Q)}
\end{equation}
Since $\tilde \alpha_k < \alpha_k$ and $v\in D^\tau_{\bar u}$ we may apply \reff{EbBk} with $v$ instead of $u-\bar u$.
Using also \eqref{secest}, we estimate 
\begin{align*}
&J'(\bar u)(u - \bar u) + J''(\bar u)(u - \bar u)^2=J'(\bar u)(v+w) + J''(\bar u)(v+w)^2\\
&\geq J'(\bar u)(v)+J'(\bar u)(w)+ J''(\bar u)(v)^2+J''(\bar u)(w)^2+2J''(\bar u)(w,v)\geq \gamma_k \|z_{\bar u, v} \|_{L^2(Q)}^{k} \| v\| _{L^1(Q)}^{2-k}+\tau \|w\|_{L^1(Q)}\\
&-3C_0(\|z_{\bar u, w}\| _{L^\infty(Q)}+\|z_{\bar u, v}\| _{L^\infty(Q)})\| w \|_{L^1(Q)} \geq \gamma_k \|z_{\bar u, v} \|_{L^2(Q)}^{k} \|v\|_{L^1(Q)}^{2-k}+\frac{\tau}{2} \|w\|_{L^1(Q)}.
\end{align*}
In the last inequality we use that by choosing $\tilde \alpha_k>0$ sufficiently small we may ensure that
\begin{align*}
&\tau-3C_0(\|z_{\bar u, w}\| _{L^\infty(Q)}+\|z_{\bar u, v}\| _{L^\infty(Q)})\geq \tau -3C_0C\max\{ \tilde \alpha,\tilde \alpha^{\frac{1}{r}}\}\geq \frac{\tau}{2}.
\end{align*}
This is implied by the inequalities
$\|z_{\bar u,w}\| _{L^\infty}$, $\|z_{\bar u,v}\| _{L^\infty(Q)}\le C_r\tilde \alpha_k^{\frac{1}{r}}$
resulting from Lemma \ref{vandw}.
Further, we find
\begin{align*}
\| w\|_{L^1(Q)}\ge \left\{\begin{array}{l} \frac{1}{2M_{\mathcal U}}\| w \| _{L^1(Q)}^2\\
 \frac{1}{C_r(2M_{\mathcal U})^{(1/r)}}\|z_{\bar u, w}\| _{L^2(Q)}^2,
\end{array}\right.
\end{align*}
where we used that $\|u-\bar u\|_{L^1(Q)}<2M_{\mathcal U}$ for all $u\in \mathcal U$ and
\[
\|z_{\bar u, w}\| _{L^2(Q)}^2\leq \|z_{\bar u, w} \|_{L^1(Q)} \|z_{\bar u, w}\|_{L^{\infty}(Q)} 
\leq  \| w\|_{L^1(Q)} C_r (2M_{\mathcal U})^{1/r}.
\]
For $k=0$:
\begin{align*}
J'(\bar u)(u - \bar u) + J''(\bar u)(u - \bar u)^2\geq \gamma_0 \|v\|   _{L^1(Q)}^{2}+\frac{\tau}{2M_{\mathcal U}} \| w\|_{L^1(Q)}^2&\geq \min\Big\{\gamma_0, \frac{\tau}{2M_{\mathcal U}}\Big\}(\|v\|_{L^1(Q)}^{2}+\|w\|_{L^1(Q)}^2)\\
&\geq \frac{2}{3}\min\Big\{\gamma_0,\frac{\tau}{2M_{\mathcal U}}\Big\}(\| u-\bar u\|_{L^1(Q)}^{2}.
\end{align*}
For $k=2$:
\begin{align*}
&J'(\bar u)(u - \bar u)+ J''(\bar u)(u - \bar u)^2
\geq \gamma_2 \|   z_{\bar u, v} \|_{L^2(Q)}^{2}+\frac{\tau}{2} \|w\|_{L^1(Q)}\\
&\geq \min\Big\{\gamma_2, \frac{\tau}{2C_r(2M_{\mathcal U})^{(1/r)}}\Big\}( \| z_{\bar u, v}\|_{L^2(Q)}^{2}
+\|z_{\bar u, w}\|   _{L^2(Q)}^2)\geq \frac{2}{3}\min\Big\{\gamma_2, \frac{\tau}{2C_r(2M_{\mathcal U})^{(1/r)}}\Big\}\|z_{\bar u, u-\bar u}\|_{L^2(Q)}^2.
\end{align*}
This proves that  \reff{E3.13.1} is satisfied with an appropriate number $\tilde\gamma_k$.
\end{proof}

If the control does not appear explicitly in the objective functional, we obtain a stronger result.

\begin{Corollary}
Let $m,g=0$. Then Assumption \ref{A4}$(A_2)$ is equivalent to the following condition ($\bar A_{2}$):
there exist constants $\alpha_2,\gamma_2,\tau>0$ such that
\begin{equation} \label{EbA2}
    J'(\bar u)(u - \bar u) + J''(\bar u)(u - \bar u)^2 \ge \gamma_2\|   z_{\bar u,u - \bar u}\|^{2}_{L^2(Q)}
\end{equation}
for all $u \in \Uad$ for which $(u-\bar u )\in C^{\tau}_{\bar u}$ and $\| y_u-\bar y\|_{L^\infty(Q)}<\alpha_2$.
\end{Corollary}

\begin{proof}
It is obvious that $(A_2)$ implies $(\bar A_{2})$. For the reverse, if $u-\bar u \in  C^{\tau}_{\bar u}$ the estimate holds trivially. 
We need to consider the cases $u-\bar u\notin G^{\tau}_{\bar u}$ and 
$u-\bar u\notin D^{\tau}_{\bar u}\text{ with }u-\bar u\in G^{\tau}_{\bar u}$. For the first, we argue as follows. 
Since $u-\bar u\notin G^{\tau}_{\bar u}$  it holds
\begin{align*}
J'(\bar u)(u-\bar u)+J''(\bar u)(u-\bar u)&
> \tau \|z_{\bar u,u-\bar u}\|_{L^1(Q)}\ge \frac{\tau}{2C_r\vert Q\vert^{\frac{1}{r}}M_{\mathcal U}}\|z_{\bar u,u-\bar u} \|_{L^2(Q)}^2.
\end{align*}
For the second case $u-\bar u\in G^{\tau}_{\bar u}$ and $u-\bar u\notin D^{\tau}_{\bar u}$, let $\tilde \alpha>0$ 
be smaller than $\alpha_2$, so that \reff{EbA2} and the prerequisite of Lemma \ref{vandw} is satisfied. 
We define $w, v$ as in Lemma \ref{vandw}. By the choice of $\alpha_2$, Lemma \ref{vandw} gives 
the existence of a constant $C>0$ such that $\|z_{\bar u,u-\bar u}\|_{L^\infty}<\alpha_2$ implies 
\begin{equation*}
\max\{ \|z_{\bar u,w}\|_{L^{\infty}(Q)}, \ \|z_{\bar u,v}\|_{L^{\infty}(Q)}\}<C\max\{ \alpha_2,\alpha_2^{\frac{1}{r}}\}.
\end{equation*}
Now we can proceed by the same arguments as in Corollary \ref{equi1}
\begin{align*}
J'(\bar u)(u - \bar u) + J''(\bar u)(u - \bar u)^2&=J'(\bar u)(v+w) + J''(\bar u)(v+w)^2 \geq \gamma_2 \|   z_{\bar u, v} \|_{L^2(Q)}^{2}+\frac{\tau}{2} \| w\|   _{L^1(Q)}.
\end{align*}

Finally, we use the estimate 
\[
\|z_{\bar u, w}\| _{L^2(Q)}^2\leq \|z_{\bar u, w} \|_{L^1(Q)} \|z_{\bar u, w}\|_{L^{\infty}(Q)} 
\leq  \|w \|_{L^1(Q)}C_r(2M_{\mathcal U})^{(1/r)}
\]
to find
\begin{align*}
J'(\bar u)(u - \bar u) &+ J''(\bar u)(u - \bar u)^2
\geq \gamma_2 \|z_{\bar u, v} \|_{L^2(Q)}^{2}+\frac{\tau}{2C_r(2M_{\mathcal U})^{(1/r)}} \|w\|   _{L^1(Q)}\\
&\geq \min\Big \{\gamma_2, \frac{\tau}{2C_r(2M_{\mathcal U})^{(1/r)}}\Big\}( \|z_{\bar u, v} \|_{L^2(Q)}^{2}
+\|    z_{\bar u, w}\|_{L^2(Q)}^2)\\
&\geq \min\Big \{\gamma_2, \frac{\tau}{2C_r(2M_{\mathcal U})^{(1/r)}}\Big\}( \|z_{\bar u, u-\bar u} \|_{L^2(Q)}^{2},
\end{align*}
for all $(u-\bar u )\in C^{\tau}_{\bar u}$ with $\| y_u-\bar y\|_{L^\infty(Q)}<\alpha_2$.
\end{proof}

\section{Strong metric H\"older subregularity and auxiliary results} \label{SSR}

We study the strong metric H\"older subregularity property (SMHSr) of the optimality map. This is an extension 
of the strong metric subregularity property (see, \cite[Section 3I]{DR} or \cite[Section 4]{CDK}) dealing with Lipschitz
stability of set-valued mappings. The SMHSr property is especially relevant to the parabolic setting where 
Lipschitz stability may fail.

\subsection{The optimality mapping}

We begin by defining some operators used to represent the optimality map in a more convenient way. 
This is done analogously to 
\cite[Section 2.1]{DJV2022}.
Given the initial data $y_0$ in \eqref{see1}, we define the set
\[
D(\mathcal L):=\Big\{ y\in W(0,T)\cap L^{\infty}(Q)\Big \vert \ \Big(\frac{d}{dt}+\mathcal A\Big)y \in L^{r}(Q), y(\cdot,0)=y_0 \Big\}.
\]
To shorten notation, we define $\mathcal L:D(\mathcal L)\to L^{r}(Q)$ by $\mathcal L:=\frac{d}{dt}+\mathcal A$.
Additionally, we define the operator 
$\mathcal L^*:D(\mathcal L^*)\to L^{r}(Q)$ by $\mathcal L^*:=(-\frac{d}{dt}+\mathcal A^*)$, where
\[
D(\mathcal L^*):=\Big\{ p\in W(0,T)\cap L^{\infty}(Q)\Big \vert  \Big(-\frac{d}{dt}+\mathcal A^*\Big)p \in L^{r}(Q), p(\cdot,T)=0 \Big\}.
\]

With the operators $\mathcal L$ and $\mathcal L^*$, we recast the semilinear state equation \reff{see1}
and the linear adjoint equation \reff{E3.8} in a short way: 
\[
\mathcal L y=u-f(\cdot,y)
\]
\[
\mathcal L^* p=L_y(\cdot,y_u,u)-pf_y(\cdot,y_u)=\frac {\partial H}{\partial y}(\cdot,y_u,p,u).
\]

The normal cone to the set $\mathcal{U}$ at $u \in L^1(Q)$ is defined in the usual way:
\begin{equation*} 
      N_{\mathcal{U}}(u):= \left\{ \begin{array}{cl}
    \big\{\nu \in L^{\infty}(Q)\big\vert  \ \int_Q \nu (v-u)\dx\dt\le 0 \ \ \forall v \in \mathcal U \big\} & \mbox{ if } u \in \mathcal{U}, \\
           \emptyset  & \mbox{ if } u \not\in \mathcal{U}.
   \end{array}  \right. 
\end{equation*}

The first order necessary optimality condition for problem \eqref{ocp1}-\eqref{const} in Theorem \ref{pontryagin}
can be recast as 
\begin{align}\label{s6}
\left\{ \begin{array}{cll}
0&=&\mathcal L y+f(\cdot,y)-u,\\
0&=&\mathcal L^*p- \frac {\partial H}{\partial y}(\cdot,y,p,u),\\   
0&\in& H_{u}(\cdot,y,p) + N_{\mathcal U}(u).
\end{array} \right.
\end{align}
For (\ref{s6}) to make sense, a solution $(y,p,u)$ must satisfy $y\in D(\mathcal L)$, $p\in D(\mathcal L^*)$ and 
$u\in\mathcal U$. 
For a local solution $\bar u\in\mathcal U$ of problem (\ref{ocp1})-(\ref{const}), by Theorem \ref{pontryagin}, 
the triple $(y_{\bar u},p_{\bar u},\bar u)$ is a solution of (\ref{s6}).
We define the sets
\begin{align} \label{EYZ}
	\mathcal Y:=D(\mathcal L)\times D(\mathcal L^*)\times\mathcal U\quad\text{and}\quad \mathcal Z:=
L^2(\Omega)\times L^2(\Omega)\times L^\infty(\Omega),
\end{align}
and consider the set-valued mapping $\Phi:\mathcal Y\twoheadrightarrow\mathcal Z$ given by
\begin{align}\label{optmapping}
\Phi\left( \begin{array}{c}
y \\
p \\
u
\end{array} \right) :=\left( \begin{array}{c}
\mathcal Ly + f(\cdot,y)-u \\
\mathcal L^*p- \frac {\partial H}{\partial y}(\cdot,y,p,u) \\
\frac {\partial H}{\partial u}(\cdot,y,p,u)+ N_{\mathcal U}(u)
\end{array} \right).
\end{align}
With the abbreviation $\psi := (y,p,u)$, the system (\ref{s6}) can be rewritten as the inclusion $0\in\Phi(\psi)$.
Our goal is to study the stability of system (\ref{s6}), or equivalently, the stability of the solutions of the inclusion 
$0\in\Phi(\psi)$ under perturbations. 
For elements $\xi,\eta\in L^r(\Omega)$ and $\rho\in L^\infty(\Omega)$ we consider the perturbed system
\begin{align}\label{s1per}
\left\{ \begin{array}{cll}
\xi&=&-\mathcal Ly+f(\cdot,y) - u,\\
\eta&=&-\mathcal Lp+\frac {\partial H}{\partial y}(\cdot,y,p,u),\\
\rho&\in&\frac {\partial H}{\partial u}(\cdot,y,p) + N_{\mathcal U}(u),
\end{array} \right.
\end{align}
which is equivalent to the inclusion $\zeta := (\xi,\eta,\rho)\in \Phi(\psi)$.

\begin{Definition}
The mapping $\Phi:\mathcal Y\twoheadrightarrow\mathcal Z$ is called the {\em optimality mapping} of the optimal 
control problem (\ref{ocp1})-(\ref{const}).
\end{Definition}

\begin{Theorem}
For any perturbation $\zeta:=(\xi,\eta,\rho)\in L^r(Q)\times L^r(Q)\times L^\infty(Q)$ there exists 
a triple $\psi := (y,p,u)\in \mathcal Y$ such that $\zeta\in \Phi(\psi)$.
\end{Theorem}

\begin{proof}
We consider the optimal control problem 
\[
\min_{u \in \mathcal U} \Big\{\mathcal J (u)+\int_Q \eta y \text{ dxdt}-\int_Q \rho u  \text{ dxdt}\Big\},
\]
subject to
\begin{align*}
\left\{\begin{array}{l}
\mathcal L y+f(x,t,y) = u + \xi\  \text{ in }\ Q,\\
y=0 \ \text{ on }\ \Sigma,\ y(\cdot,0)=y_0 \text{ in } \Omega.
\end{array} \right.
\end{align*}
Under assumptions \ref{exist.1} and \ref{exist.2}, we have by standard arguments the existence of
a global solution $\tilde u$. Then $\tilde u$ and the corresponding state $y_{\tilde u}$ and adjoint 
state $p_{\tilde u}$ satisfy \eqref{s1per}.
\end{proof}

The following extension of the previous theorem can be proved along the lines of \cite[Theorem 4.12]{DJV2022}.

\begin{Theorem}
Let Assumption \ref{A4}$(A_0)$ hold. For each $\varepsilon > 0$ there exists $\delta > 0$ 
such that for every $\zeta \in B_{Z} (0;\delta)$
there exists $\psi \in B_{Y}(\bar \psi; \varepsilon)$ satisfying the inclusion $\zeta \in \Phi(\psi)$.
\end{Theorem}

\subsection{Strong metric H\"older subregularity: main result}
\label{SmHsr}

This subsection contains one of the main results in this paper: estimates of the difference between the solutions of the perturbed 
system \reff{s1per} and a reference solution of the unperturbed one, \reff{s6}, by the size of the perturbations.
This will be done using the notion of {\em strong metric H\"older subregularity} introduced in the next paragraphs.
 
Given a metric space $(\mathcal X,d_{\mathcal X})$, we denote by $B_{\mathcal X}(c,\alpha)$ the closed ball of center 
$c\in\mathcal X$ and radius $\alpha>0$.
The spaces $\mathcal Y$ and $\mathcal Z$, introduced in \reff{EYZ}, are endowed with the metrics 
\begin{align} \label{Enzeta}
	d_{\mathcal Y}(\psi_1,\psi_2)&:=\|   y_1-y_2\|   _{L^2(Q)}+\|   p_1-p_2\|   _{L^2(Q)}+\|   u_1-u_2\|   _{L^1(Q)},\\
	\nonumber d_{\mathcal Z}(\zeta_1,\zeta_2)&:=\|   \xi_1-\xi_2\|   _{L^{2}(Q)}+\|   \eta_1-\eta_2\|   _{L^{2}(Q)}+
      \|   \rho_1-\rho_2\|   _{L^\infty(Q)},
\end{align}
where $\psi_i=(y_i,p_i,u_i)$ and $\zeta_i=(\xi_i,\eta_i,\rho_i)$, $i\in\{1,2\}$.
From now on, we denote $\bar\psi:=(y_{\bar u},p_{\bar u},\bar u)$ to simplify notation. 

\begin{Definition}\label{Dsmsr}
Let $\bar \psi$ satisfy $0 \in \Phi(\bar \psi)$.
We say that the optimality mapping $\Phi:\mathcal Y\twoheadrightarrow \mathcal Z$ is 
{\em strongly metrically H\"older subregularity} (SMHSr) at $(\bar\psi,0)$ 
with exponent $\theta>0$ if there exist positive numbers $\alpha_1,\alpha_2$ and $\kappa$ such that 
	\begin{align*}
	d_{\mathcal Y}(\psi,\bar\psi)\le\kappa d_{\mathcal Z}(\zeta,0)^\theta
	\end{align*}
	for all $\psi\in B_{\mathcal Y}(\bar \psi\text{; }\alpha_1)$ and $\zeta\in B_{\mathcal Z}(0\text{; }\alpha_2)$ 
satisfying $\zeta\in\Phi(\psi)$.
\end{Definition}

Notice that applying the definition with $\zeta = 0$ we obtain that $\bar \psi$ is the unique solution 
of the inclusion $0 \in \Phi(\psi)$ in $ B_{\mathcal Y}(\bar \psi\text{; }\alpha_1)$. In particular, $\bar u$ is a strict  
local minimizer for problem \eqref{ocp1}-\eqref{const}.

In the next assumption we introduce a restriction on the set of admissible perturbations, call it $\Gamma$, which is
valid for the  remaining part of this section.

\begin{Assumption}  \label{Ape}
For a fixed positive constant $C_{pe}$, the admissible perturbation $\zeta = (\xi,\eta,\rho) \in \Gamma \subset \mathcal Z$ satisfy the restriction 
\begin{equation} \label{pertbound}
         \|\xi \|_{L^r(Q)}\leq C_{pe} .
\end{equation}
\end{Assumption}

For any $u \in \mathcal U$ and $\zeta \in \Gamma$ we denote by $(y_u^\zeta, p_u^\zeta,u)$ a solution 
of the first two equations in \reff{s1per}.
Using \reff{semilin} in Theorem~\ref{estsemeq} we obtain the existence of a constant $K_y$ such that
\begin{equation} \label{E3.15}
        \|y_u^\zeta\| _{L^\infty(\bar{Q})}  \leq K_y \quad \forall u \in \mathcal U 
         \;\; \forall \zeta \in \Gamma.
\end{equation}
Then for every $u \in {\mathcal U}$, every admissible disturbance $\zeta$, and the corresponding solution $y$ of
the first equation in \reff{s1per} it holds that $(y_u^\zeta(x,t),u(x,t)) \in R := [- K_y , K_y] \times [u_a, u_b]$.

\begin{Remark} \label{Rlip}
We apply the local properties in  Assumption \ref{exist.2} to the interval $[-K_y, K_y]$, and denote further by $\bar C$ a constant
that majorates the bounds and the Lipschitz constants of $f$ and $L_0$ and their first and second derivatives with respect to $y \in [-K_y, K_y]$.
\end{Remark}

By increasing the constant $K_y$, if necessary, we may also estimate the adjoint state:
\begin{equation*}
        \|p_u^\zeta\| _{L^\infty(\bar{Q})}  \leq K_y (1 + \| \eta \|_{L^r}(Q)) \quad \forall u \in \mathcal U 
         \;\; \forall \zeta \in \Gamma.
\end{equation*}
This follows from Theorem \ref{mainex} with $\alpha = - \frac{\partial f}{\partial y}(x,t,y_u^\zeta)$ and with
$\frac{\partial L}{\partial y}(x,t,y_u^\zeta,u)$ at the place of $u$.

We need some technical lemmas before stating our main result.

\begin{Lemma} Let $u\in \mathcal U$ be given and $v, \eta\in L^r(Q)$, $\xi\in L^\infty(Q)$. Consider solutions 
$y_u, y_u^\xi, p_u$ and $p_u^\eta$ of the equations
\begin{align}\label{s1}
\left\{ \begin{array}{cll}
\mathcal Ly+f(\cdot,y)&=&u+\xi,\\
\mathcal Ly+f(\cdot,y)&=&u,
\end{array} \right. \
\left\{ \begin{array}{cll}
\mathcal L^*p-\frac {\partial H}{\partial y}(\cdot,y_u^\xi,p,u)&=&\eta,\\
\mathcal L^*p-\frac {\partial H}{\partial y}(\cdot,y_u,p,u)&=&0,
\end{array} \right.
\end{align}
and solutions $z_{\bar u,v}^\xi$, $z_{\bar u, v}$ of
\begin{align}\label{eqz}
\left\{ \begin{array}{cll}
\mathcal Lz+f_y(\cdot,y^\xi_u)z&=&v,\\
\mathcal Lz+f_y(\cdot,y_u)z&=&v.
\end{array} \right. \
\end{align}
There exists constants $\beta_i>0$, $i\in \{1,2\}$,
independent of $\zeta \in \Gamma$, 
such that the following inequalities hold
	\begin{align}
&\|y^\xi_u - y_u\| _{L^2(Q)} \le \hat C\|\xi\|_{L^2(Q)},\label{E4.3.2}\\
&\| z^\xi_{u,v} - z_{u,v}\|_{L^2(Q)} \le \beta_1
 \|\xi\|_{L^{r}(Q)}\|z_{u,v}\| _{L^2(Q)},\label{E4.4}\\
 &\| z^\xi_{u,v} - z_{u,v}\|_{L^s(Q)} \le \beta_1
 \|\xi\|_{L^{2}(Q)}\|z_{u,v}\| _{L^2(Q)},\label{E4.4b}\\
&\| p^\eta_u - p_u\|_{2}\le \beta_2(\|\xi\|_{L^2(Q)}+\| \eta\| _{L^2(Q)}),\label{E4.5.a2}
	\end{align}
where $\hat C$ is the constant given in \eqref{wl2} and $ s\in [1,\frac{n+2}{n})$.
	\label{T4.1}
\end{Lemma}

\begin{proof}
	Subtracting the state equations in \eqref{s1} and using the mean value theorem we obtain
	\[
	\frac{d}{dt}(y^\xi_u - y_u)+\mathcal{A}(y^\xi_u - y_u) + \frac{\partial f}{\partial y}(x,t,y_\theta)(y^\xi_u - y_u) = \xi.
	\]
	Then, \eqref{wl2} implies \eqref{E4.3.2}. To prove \eqref{E4.4} we subtract the equations \eqref{eqz} satisfied by 
$z^\xi_{u,v}$ and $z_{u,v}$ to obtain
	\[
	\frac{d}{dt}(z^\xi_{u,v} - z_{u,v})+\mathcal{A}(z^\xi_{u,v} - z_{u,v}) +
     \frac{\partial f}{\partial y}(x,t,y^\xi_u)(z^\xi_{u,v} - z_{u,v}) = \Big[\frac{\partial f}{\partial y}(x,t,y_u) - \frac{\partial f}{\partial y}(x,t,y^\xi_u)\Big]z_{u,v}.
	\]
Now, using \eqref{wl2}, the mean value theorem, and \eqref{pertbound} we obtain
	\begin{align*}
		&\|   z^\xi_{u,v} - z_{u,v}\|   _{L^2(Q)} \le \hat C\Big\|   \Big[\frac{\partial f}{\partial y}(x,t,y_u) - 
       \frac{\partial f}{\partial y}(x,t,y^\xi_u)\Big]z_{u,v}\Big\|   _{L^2(Q)} \le \hat C\bar{C}\|   (y^\xi_u - y_u)z_{u,v}\|   _{L^2(Q)}\\
&\le \hat C\bar{C}\|   y^\xi_u - y_u\|   _{L^\infty(Q)}\|   z_{u,v}\|   _{L^2(Q)} \le C_r\hat C\bar{C}\| \xi\|_{L^r(Q)}\| z_{u,v}\|   _{L^2(Q)}.
	\end{align*}
	The proof for estimate \eqref{E4.4b} follows by the same argumentation but using \eqref{aeqestLs}. We denote by $\beta_1>$ the maximum of the constants appearing in the estimate above and its analog for \eqref{E4.4b}.
Finally, we subtract the adjoint states and employ the mean value theorem to find
	\begin{align*}
	&-\frac{d}{dt}(p^\eta_u - p_u)+\mathcal{A}^*(p^\eta_u - p_u) + \frac{\partial f} {\partial y}(x,t,y^\xi_u)(p^\eta_u - p_u) \\
	&=  \frac{\partial^2 L} {\partial y^2}(x,t,y_\theta)(y^\xi_u-y_u)+\frac{\partial^2 f} {\partial y^2}(x,t,y_\theta)(y^\xi_u-y_u)p_u+\eta.
	\end{align*}
The claim follows using \eqref{wl2}, \eqref{unibound}, and \eqref{E3.15} to estimate
	\begin{align*}
	\|p^\eta_u - p_u\|_{L^2(Q)}\leq (\hat C^2\bar{C}+M_{\mathcal U}\hat C^2\bar C+\hat C)(\| \xi\|_{L^2(Q)}+\| \eta\|_{L^2(Q)}).
	\end{align*}
\end{proof}

 \begin{Lemma}
Let $s\in [1,\frac{n+2}{n})\cap[1,2]$. Let $u \in\mathcal U$ and let $y_u$, $p_u$ be the corresponding state and adjoint state. 
Further, let $y_u^\zeta$ and $p^\zeta_u$ be solutions to the perturbed state and adjoint equation in \reff{s1per}
for the control $u$. There exist constants $C,\tilde C>0$,
independent of $\zeta \in \Gamma$, such that for $v\in \mathcal U$, the following estimates hold.
\begin{enumerate}
\item For $m=0$ in \reff{genobj}:
\begin{align}
\Big\vert  \int_Q&\Big(\frac {\partial H}{\partial u}(x,t,y_u,p_u)-
\frac {\partial H}{\partial u}(x,t,y^{\zeta}_u,p^{\zeta}_u)\Big)(v-u)\dx\dt\Big\vert\nonumber  \\
&\le  C(\|\xi\| _{L^2(Q)}+\|\eta\| _{L^2(Q)})\| z_{u,u - v}\| _{L^2(Q)}\\
&\le \tilde C(\| \xi\|_{L^2(Q)}+\|\eta\| _{L^2(Q)})\| v-u\|_{L^1(Q)}^{\frac{3s-2}{2s}}.
\label{esti2}
\end{align}
\item For a general $m\in \mathbb R$:
\begin{align}
\Big\vert  \int_Q&\Big(\frac {\partial H}{\partial u}(x,t,y_u,p_u)-
\frac {\partial H}{\partial u}(x,t,y^{\zeta}_u,p^{\zeta}_u)\Big)(v-u)\dx\dt\Big\vert\le \tilde C(\| \xi\|_{L^r(Q)}+\|\eta\| _{L^r(Q)})\| v-u\|_{L^1(Q)}.\label{Lr}
\end{align}
\end{enumerate}
\label{Impest}
\end{Lemma}

\begin{proof}
We consider the first case, $m=0$. We begin with integrating by parts
\begin{align}
&\Big\vert  \int_Q\Big(\frac {\partial H}{\partial u}(x,t,y_u,p_u)-\frac {\partial H}{\partial u}(x,t,y^{\zeta}_u,p^{\zeta}_u)\Big)(v-u)\dx\dt\Big\vert  \nonumber\\
&\le\Big\vert  \int_Q\Big[\frac{\partial L_0}{\partial y}(x,t,y_u)z_{u,u-v}-\frac{\partial L_0}{\partial y}(x,t,y^{\zeta}_u)z_{u,u-v}^{\zeta}\Big]\dx\dt\Big\vert+\Big\vert  \int_Q z_{u,u-v}^{\zeta} \eta \dx \dt\Big\vert \nonumber \\
&\le\int_Q\Big\vert  \frac{\partial L_0}{\partial y}(x,t,y_u)-\frac{\partial L_0}{\partial y}(x,t,y^{\zeta}_u)\Big\vert  \Big\vert z_{u,u-v}\Big\vert  \dx\dt+\int_Q\Big\vert  \frac{\partial L_0}{\partial y}(x,t,y^{\zeta}_u)+\eta \Big\vert   \Big\vert  z_{u,u-v}-z_{u,u-v}^{\zeta}\Big\vert  \dx\dt\nonumber\\
& +\Big\vert  \int_Q \eta z_{u,u-v}\dx\dt\Big\vert  =I_1+I_2+I_3.\nonumber
\end{align}
For the first term we use the H\"older inequality, the mean value theorem, \eqref{aeqestLs}, \eqref{unibound},  and \eqref{E4.3.2} to estimate
\begin{align*}
I_1& \leq \int_Q\Big\vert  \frac{\partial L_0}{\partial y}(x,t,y_u) - \frac{\partial L_0}{\partial y}(x,t,y_{u}^\zeta)\Big\vert   \vert  z_{u,u - v}\vert  \dx\dt\\
& \le \bar{C}\|   y^\zeta_{u} - y_{u}\|   _{L^2(Q)}\|   z_{u,u - v}\|   _{L^2(Q)} \le \bar{C}\hat C\|   \xi\|   _{L^2(Q)}\|   z_{u,u - v}\|   _{L^2(Q)} \\
&\le \bar{C}\hat CC_{s'}^{1+\frac{2-s}{2}}(2M_{\mathcal U})^{\frac{(s'-1)(2-s)}{2s'}}\|\xi\| _{L^2(Q)}\| u-v\| _{L^1(Q)}^{1+\frac{s-2}{2s}}.
\end{align*}
Here we used that by Theorem \ref{mainex} and Lemma \ref{aeqestLs} it holds
\begin{align*}
\|z_{u,u - v}\|_{L^2(Q)}&\leq \|z_{u,u - v}\|_{L^\infty(Q)}^\frac{2-s}{2}\|z_{u,u - v}\| _{L^s(Q)}^{\frac{s}{2}}\leq C_{s'}^{1+\frac{2-s}{2}}(2M_{\mathcal U})^{\frac{(s'-1)(2-s)}{2s'}}\| u-v\|_{L^1(Q)}^{\frac{2-s}{2s'}+\frac{s}{2}},
\end{align*}
and noticing that $\frac{2-s}{2s'}+\frac{s}{2}=1-\frac{2-s}{2s}$.
The second term is estimated by using \eqref{unibound}, H\"older's inequality, and \eqref{E4.4}:
\begin{align*}
I_2 &\leq  \int_Q\Big\vert  \frac{\partial L_0}{\partial y}(x,t,y^\zeta_u)+\eta\Big\vert   \Big\vert  z^\zeta_{u,u-v} - z_{u,u -v}\Big\vert  \dx\dt\\
&\le \beta_1\max\{d_1,\vert Q\vert^{\frac{1}{r}}C_{pe}\}(\|\xi\| _{L^2(Q)}+\|\eta\|_{L^2(Q)})\| z_{u,u - v}\|_{L^2(Q)}\\
&\le d_2(\|\xi\| _{L^2(Q)}+\|\eta\|_{L^2(Q)})\|u- v\|_{L^1(Q)}^{1+\frac{s-2}{2s}},
\end{align*}
where 
$d_1:=\|\psi_{M_{\mathcal U}}\| _{L^{s'}(Q)}$ and $d_2:=\beta_1\max\{d_1,\vert Q\vert^{\frac{1}{r}}C_{pe}\} C_{s'}^{1+\frac{2-s}{2}}(2M_{\mathcal U})^{\frac{(s'-1)(2-s)}{2s'}}$.
For last term we estimate
\begin{equation*}
  I_3 \leq \Big\vert  \int_Q \eta z_{u,u-v}\dx\dt\Big\vert  \le \|z_{u,u-v}\|_{L^2(Q)}\|\eta\|_{L^2(Q)}.
\end{equation*}
We prove the second case \eqref{Lr}. By applying \eqref{clr} and arguing as in the proof of \eqref{E4.3.2} and \eqref{E4.5.a2} but for $r$, we infer the existence of a constant, again denoted by $\tilde C>0$, such that:
\begin{align*}
&\Big\vert  \int_Q\Big(\frac {\partial H}{\partial u}(x,t,y_u,p_u)-\frac {\partial H}{\partial u}(x,t,y^{\zeta}_u,p^{\eta}_u)\Big)(v-u)\dx\dt\Big\vert \nonumber \\
&=\Big\vert  \int_Q \Big[p_u-p^{\eta}_u+m(y_u-y^{\zeta}_u)\Big](v-u)\dx\dt\Big\vert\nonumber\\
&\leq \|p_u-p^{\eta}_u+m(y_u-y^{\zeta}_u)\|_{L^\infty(Q)}\|u-\bar u\|_{L^1(Q)}\\
& \le \tilde C(\|\xi\|_{L^r(Q)}+\| \eta\|_{L^r(Q)})\|v-u\| _{L^1(Q)}.
\end{align*}
\end{proof}

The main result in the paper follows.

\begin{Theorem}\label{Ssr}
Let assumption
\ref{A4}(A$_0$) be fulfilled for the reference solution 
$\bar \psi = (\bar y,\bar p, \bar u)$ of $0 \in \Phi(\psi)$.
Then the mapping $\Phi$ is strongly metrically H\"older subregular at 
$(\bar \psi,0)$. More precisely, for every $\varepsilon \in (0, 1/2]$ there exist $\alpha_n > 0$ and $\kappa_n$ 
(with $\alpha_1$ and $\kappa_1$ independent of $\varepsilon$) such that 
for all $\psi \in \mathcal Y$ with $\| u-\bar u\|_{L^1(Q)} \leq \alpha_n$ and $\zeta\in \Gamma$ satisfying $\zeta\in\Phi(\psi)$,
the following inequalities are satisfied.
\begin{enumerate}
\item In the case $m=0$ in \reff{genobj}:
\begin{align} \label{EEu}
&\!\!\!\!\!\!\!\!\!\!\!\! \| \bar u-u \|_{L^1(Q)}\le \kappa_n \Big(\|\rho \|_{L^{\infty}(Q)}+\|\xi\|_{L^2(Q)}+\|\eta\| _{L^2(Q)}\Big)^{\theta_0},\\
& \!\!\!\!\!\!\!\!\!\!\!\! \label{EEy}  \|y_{\bar u}-y^{\zeta}_u\|_{L^2(Q)}+\| p_{\bar u}-p^{\zeta}_u\|_{L^2(Q)}
\le \kappa_n \Big(\|\rho \|_{L^{\infty}(Q)}+\|\xi\|_{L^2(Q)}+\|\eta\| _{L^2(Q)}\Big)^{\theta},
\end{align}
where 
\begin{align} \label{Ethetas}
     &  \theta_0 = \theta = 1 & \hspace*{-2cm} \mbox{ if } \;  n = 1, & \qquad\\
    & \theta_0 = \theta = 1 - \varepsilon & \hspace*{-2cm} \mbox{ if } \; n= 2,& \qquad \\
  &  \theta_0 = \frac{10}{11}  - \varepsilon, \;\;  \theta = \frac{9}{11}  - \varepsilon & \hspace*{-2cm}  \mbox{ if } \; n= 3.& \qquad
\end{align}
\item In the general case $m\in \mathbb R$:
\begin{align} \label{EEum}
&\!\!\!\!\!\!\!\!\!\!\!\!\| \bar u-u \|_{L^1(Q)}\le \kappa_n \Big(\|\rho \|_{L^{\infty}(Q)}+\|\xi\|_{L^r(Q)}+\|\eta\| _{L^r(Q)}\Big),\\
& \!\!\!\!\!\!\!\!\!\!\!\! \label{EEym}  \|y_{\bar u}-y^{\zeta}_u\|_{L^2(Q)}+\| p_{\bar u}-p^{\zeta}_u\|_{L^2(Q)}
\le \kappa_n \Big(\|\rho \|_{L^{\infty}(Q)}+\|\xi\|_{L^r(Q)}+\|\eta\| _{L^r(Q)}\Big)^{\theta_0}.
\end{align}
\end{enumerate}
\end{Theorem}

\begin{proof}
We begin with the proof for $m=0$. We select $\alpha_1 <\tilde\alpha_0$ according to Lemma \ref{good}. Let $\zeta=(\xi,\eta,\rho)\in\mathcal Z$ and $\psi=(y^{\zeta}_u,p^{\zeta}_u,u)$ with $\| u-\bar u\|_{L^1(Q)}\leq\alpha_1$ such that  $\zeta\in\Phi(\psi)$, i.e. 
\begin{align*}
\left\{ \begin{array}{cll}
\xi&=&\mathcal Ly^{\zeta}_u+f(\cdot,\cdot,y^{\zeta}_u) - u,\\
\eta&=&\mathcal L^*p_u^{\zeta}-\frac {\partial H}{\partial y}(\cdot,y^{\zeta}_u,p_u^{\zeta},u),\\
\rho&\in& \frac {\partial H}{\partial u}(\cdot,y^{\zeta}_u,p_u^{\zeta}) + N_{\mathcal U}(u).
\end{array} \right.
\end{align*}
Let $y_u$ and $p_u$ denote the solutions to the unperturbed problem with respect to $u$, i.e.
\[
u=\mathcal Ly_u+f(\cdot,\cdot,y_u)\textrm{ and } 0=\mathcal L^*p_u-\frac {\partial H}{\partial y}(\cdot,y_u,p_u,u).
\]
By Lemma \ref{T4.1}, there exists $\hat C,\beta_2>0$ independent of $\psi$ and $\zeta$ such that
\begin{align}\label{mt1}
\|y^{\zeta}_u-y_{u}\|_{L^2(Q)}&+\|p^{\zeta}_u-p_{u}\| _{L^2(Q)}\le (\hat C+\beta_2)\Big(\|\xi\|_{L^2(Q)}+\|\eta\|_{L^2(Q)}\Big).
	\end{align}
By the definition of the normal cone, $\rho\in \frac {\partial H}{\partial u}(\cdot,\cdot,y^{\zeta}_u,p^{\zeta}_u)+ N_{\mathcal U}(u)$ is equivalent to
\begin{align*}
0\ge \int_Q(\rho-\frac {\partial H}{\partial u}(\cdot,\cdot,y^{\zeta}_u,p^{\zeta}_u))(w-u)\ \ \forall w\in \mathcal U.
\end{align*}
We conclude for $w=\bar u$, 
\begin{align}
0&\ge \int_Q \frac {\partial H}{\partial u}(\cdot,\cdot,y_u,p_u)(u-\bar u)+\int_Q(\rho+\frac {\partial H}{\partial u}(\cdot,\cdot,y_u,p_u)-\frac {\partial H}{\partial u}(\cdot,\cdot,y^{\zeta}_u,p^{\zeta}_u))(\bar u-u)\nonumber\\
&\ge J'(u)(u-\bar u)-\|\rho\| _{L^{\infty}(Q)}\|\bar u-u\| _{L^1(Q)}-\Big\vert  \int_Q(\frac {\partial H}{\partial u}(\cdot,\cdot,y_u,p_u)-\frac {\partial H}{\partial u}(\cdot,\cdot,y^{\zeta}_u,p^{\zeta}_u))(\bar u-u)\dx\dt\Big\vert.\label{78}
\end{align}
By Lemma \ref{Impest}, we have an estimate on the third term. Since $\| u-\bar u\|_{L^1(Q)}<\tilde \alpha_0$, we estimate by Lemma \ref{good} and Lemma \ref{Impest}
 \begin{align*}
\|    u-\bar u\|   _{L^1(Q)}^2\tilde \gamma&\le J'(u)(u-\bar u)\le \tilde C \Big(\|   \xi\|   _{L^2(Q)}+\|   \eta\|   _{L^2(Q)}\Big) \| u-\bar u\| _{L^1(Q)}^{1+(s-2)/(2s)}+\|    \rho\|   _{L^{\infty}(Q)}\|    \bar u-u\|   _{L^1(Q)}
\end{align*}
and consequently for an adapted constant, denoted in the same way
\begin{equation*}
\|\bar u-u\|_{L^1(Q)}\le \tilde C\Big(\|\rho \| _{L^{\infty}(Q)}+\|\xi\| _{L^2(Q)}+\|\eta\| _{L^2(Q)}\Big)^{\frac{2s}{s+2}}.
\end{equation*}
To estimate the states, we use the estimate for the controls. We notice that $(2-s)/(2s')+s/2=1+(s-2)(2s)$ and obtain
\begin{align}
\|y_{\bar u}-y_u\| _{L^{2}(Q)}&\le \| y_{\bar u}-y_u \|_{L^{\infty}(Q)}^{\frac{2-s}{2}} \|y_{\bar u}-y_u\| _{L^s(Q)}^{\frac{s}{2}}\le C_r^{\frac{2-s}{2}}\|\bar u-u \|_{L^1(Q)}^{1+\frac{s-2}{2s}}\label{stc}.
\end{align}
Thus, for a constant again denoted by $\tilde C$ and with $(1+\frac{s-2}{2s})\frac{2s}{s+2}=\frac{3s-2}{2+s}$,
\begin{equation*}\label{Eststate}
\|   y_{\bar u}-y_u\| _{L^{2}(Q)}\le \tilde  C \Big(\|\xi\|_{L^2(Q)}+\|\eta\| _{L^2(Q)}+\|\rho\| _{L^{\infty}(Q)}\Big)^{\frac{3s-2}{2+s}}.
\end{equation*}
Next, we realize that by Lemma \ref{T4.1} and \eqref{Eststate}
\begin{align*}
&\|   y_{\bar u}-y_u^\zeta\|   _{L^2(Q)}\le \|   y_{\bar u}-y_u \|   _{L^2(Q)}+\|   y_{u}-y_u^\zeta \|   _{L^2(Q)}\le \max\{\tilde C,\hat C\}\Big(\|   \xi\|   _{L^2(Q)}+\|\eta\| _{L^2(Q)}+\|\rho\| _{L^{\infty}(Q)}\Big)^{\frac{3s-2}{2+s}}.
\end{align*}
Using $\|p_{\bar u}-p_u\| _{L^2(Q)}\le \hat C\| y_{\bar u}-y_u\|_{L^2(Q)}$ and \eqref{E4.5.a2}, the same estimate holds for the adjoint state
\begin{align*}
&\|p_{\bar u}-p_u^\zeta \|_{L^2(Q)}\le \| p_{\bar u}-p_u \| _{L^2(Q)}+\|   p_{u}-p_u^\zeta \|_{L^2(Q)}\le (\hat C\tilde C+\beta_2) \Big(\|   \xi\|_{L^2(Q)}+\|\eta\| _{L^2(Q)}+\|\rho\| _{L^{\infty}(Q)}\Big)^{\frac{3s-2}{2+s}},
\end{align*}
subsequently we define $\kappa:=\max\{\tilde C,\hat C\}$. Finally, we consider the case $m\neq 0$. Using estimate \ref{Lr} in \eqref{78} and arguing from that as for the case $m=0$, we infer the existence of a constant $\tilde C>0$ such that
\begin{equation*}
\| u-\bar u\|_{L^1(Q)}\leq \tilde C \Big( \|\rho \| _{L^{\infty}(Q)}+\|\xi\| _{L^r(Q)}+\|\eta\| _{L^r(Q)}\Big).
\end{equation*}
This implies under \eqref{stc} the estimate for the states and adjoint-states
\begin{align*}
&\|   y_{\bar u}-y_u^\zeta\|   _{L^2(Q)}+\|p_{\bar u}-p_u^\zeta \| _{L^2(Q)}\le \max\{\tilde C,\hat C\tilde C+\beta_2\}\Big(\|   \xi\|   _{L^2(Q)}+\|\eta\| _{L^2(Q)}+\|\rho\| _{L^{\infty}(Q)}\Big)^{1+(s-2)/(2s)}.
\end{align*}
To determine $\theta$ and $\theta_0$ we notice that the functions $s\to \frac{s-2}{2s}$ and $s\to \frac{3s-2}{2+s}$ are monotone. Inserting the value for $\frac{n+2}{2}$ for each case $n\in \{1,2,3\}$ completes the proof.
\end{proof}

To obtain results under Assumption \ref{A4} for $k\in\{1,2\}$, we need additional restrictions. 
We either don't allow perturbations $\rho$ (appearing in the inclusion in \reff{s1per}) or they need to satisfy 
\begin{equation}
           \rho\in D(\mathcal L^*).
\label{perturbeq}
\end{equation}

\begin{Theorem} \label{SsHr}
Let $m=0$ and let 
some of the assumptions 
$(A_1),(B_1)$ and $(A_2),(B_2)$ be fulfilled for the reference solution $\bar \psi = (\bar y,\bar p, \bar u)$ of $0 \in \Phi(\psi)$.
Let, in addition, the set $\Gamma$ of feasible perturbations be restricted to such $\zeta \in \Gamma$ for which 
the component $\rho$ is either zero or satisfies \eqref{perturbeq}. 
The numbers $\alpha_n$, $\kappa_n$ and $\varepsilon$ are as in Theorem \ref{Ssr}.
Then the following statements hold for $n \in \{ 1,2,3\}$:

1. Under Assumption \ref{A4}, cases $(A_1)$ and $(B_1)$, the estimations 
\begin{align*}
  &\| \bar u-u \|_{L^1(Q)}\le \kappa_n \Big(\|\mathcal L^*\rho \|_{L^{\infty}(Q)}+\|\xi\|_{L^2(Q)}+\|\eta\| _{L^2(Q)}\Big),\\
 &\|y_{\bar u}-y^{\zeta}_u\|_{L^2(Q)}+\| p_{\bar u}-p^{\zeta}_u\|_{L^2(Q)}
\le \kappa_n \Big(\|\rho \|_{L^{\infty}(Q)}+\|\xi\|_{L^2(Q)}+\|\eta\| _{L^2(Q)}\Big)^{\theta_0},
\end{align*}
with $\theta_0$ as in Theorem \ref{Ssr}, hold 
for all $u\in \mathcal U$ with $\|y_u-\bar y\|_{L^\infty(Q)}<\alpha_n$, in the case of $(A_1)$, or 
$\|u-\bar u \|_{L^1(Q)}<\alpha_n$ in the case ($B_1$), and for all $\zeta\in \Gamma$ satisfying $\zeta\in\Phi(\psi)$.

2. Under Assumption \ref{A4}, cases $(A_2)$ and $(B_2)$, the estimation
\begin{equation*}
\|\bar y-y^{\zeta}_u\|_{L^2(Q)}+\|\bar p-p^{\zeta}_u\|_{L^2(Q)} \le 
        \kappa_n \Big(  \|\xi\| _{L^2(Q)} + \|\eta\| _{L^2(Q)} + \| \mathcal L^* \rho \|_{L^{2}(Q)} \Big)
\end{equation*}
hold for all $u\in \mathcal U$ with $\|y_u-\bar y\|_{L^\infty(Q)}<\alpha_n$, in the case of $(A_2)$, or 
$\|u-\bar u \|_{L^1(Q)}<\alpha_n$ in the cases ($B_2$), and for all $\zeta\in \Gamma$ satisfying $\zeta\in\Phi(\psi)$.
\end{Theorem}

\begin{proof}
We first notice that if the perturbation $\rho$ satisfies \eqref{perturbeq}, it holds
\begin{align*}
\int_Q \rho(u-\bar u)\dx \dt&=\int_Q((\frac{d}{dt}+\mathcal A) z_{\bar u,u-\bar u}+f_y(x,t,y_{\bar u})z_{\bar u,u-\bar u})\rho\dx\dt\\
&=\int_Q((-\frac{d}{dt}+\mathcal A^*) \rho+f_y(x,t,y_{\bar u})\rho)z_{\bar u,u-\bar u}\dx\dt.
\end{align*}
Thus
\begin{align*}
&\quad\Big\vert  \int_Q \rho(u-\bar u)\dx \dt\Big\vert\le \| z_{\bar u,u-\bar u}\|_{L^2(Q)}(\| \mathcal L^* \rho\| _{L^2(Q)}+\|f_y(x,t,y_{\bar u})\| _{L^\infty(Q)}\| \rho\|_{L^2(Q)}).
\end{align*}
Under Assumption $(A_1)$, we can proceed as in the proof of Theorem \ref{Ssr} using Lemma \ref{good} and 
\eqref{esti2} in Lemma \ref{Impest}, to infer the existence of constants $\alpha,\kappa_1>0$ such that
\begin{equation*}
\|\bar u-u\|_{L^1(Q)}\le \kappa_1\Big(\|\mathcal L^* \rho \| _{L^{2}(Q)}+\|   \xi\|   _{L^2(Q)}+\| \eta\| _{L^2(Q)}\Big),
\end{equation*}
and by standard estimates the existence of a constant $\hat C>0$ and using \eqref{E2.14.2}
\begin{align*}
&\| y_{\bar u}-y_u\| _{L^2(Q)}+\| p_{\bar u}-p_u\|_{L^2(Q)}\le \hat C\|    y_{\bar u}-y_u\| _{L^2(Q)}\le 2\hat C \|z_{u,u-\bar u}\|_{L^2(Q)}\\
&\le 2\hat C\kappa_1^{\frac{2s}{s+2}}\Big(\| \mathcal L^* \rho  \| _{L^{2}(Q)}+\|\xi\| _{L^2(Q)}+\|\eta\| _{L^2(Q)}\Big)^{\frac{2s}{s+2}},
\end{align*}
for all $u\in \mathcal U$ with $\|y_u-\bar y\|_{L^\infty(Q)}<\alpha$ or $\|u-\bar u \|_{L^1(Q)}<\alpha$ depending on the assumption.
From here on, one can proceed as in the proof of Theorem \ref{Ssr} and define the final constant $\kappa>0$ and the exponent $\theta_0$ accordingly.
Finally, by similar reasoning, under Assumption $(A_2)$ with Lemma \ref{good} and Lemma \ref{Impest}, one obtains the existence of a constant $\kappa>0$ such that
\begin{align*}
&\| y_{\bar u}-y_u\|_{L^2(Q)}+\| p_{\bar u}-p_u\|_{L^2(Q)}\le \kappa\Big(\|\mathcal L^* \rho  \|_{L^{2}(Q)}+\|\xi\| _{L^2(Q)}+\|\eta\|_{L^2(Q)}\Big),
\end{align*}
for all $u\in \mathcal U$ with $\|y_u-\bar y\|_{L^\infty(Q)}<\alpha$ or $\|u-\bar u \|_{L^1(Q)}<\alpha$.
Again, proceeding as in Theorem \ref{Ssr} and increasing the constant $\kappa$ if needed, proves the claim.
\end{proof}

\begin{Remark} \label{Rnonlin}
Theorems \ref{Ssr} and \ref{SsHr} concern  perturbations which are functions of $x$ and $t$ only. On the other hand, 
\cite[Theorem ]{CDK} suggests that SMHSr implies a similar stability property under classes of perturbations that depend 
(in a non-linear way) on the state and control. This fact will be used and demonstrated in the next section.   
\end{Remark}

\section{Stability of the optimal solution} \label{Stability}

In this section we obtain stability results for the optimal solution under non-linear perturbations in the objective functional.
Namely, we consider a disturbed problem
\begin{equation} \label{newequa}
\mbox{\rm (P$_{\zeta}$)}\xspace \  \min_{u \in \Uad} J_\zeta(u) := \int_Q[L(x,t,y(x,t),u(x,t))+ \eta(x,t,y_u(x,t),u(x,t))]\dx\dt,
\end{equation}
subject to
\begin{align} \label{Eyxi}
\left\{\begin{array}{l}
\frac{dy}{dt}+\mathcal A y+f(x,t,y) = u + \xi \  \text{ in }\ Q,\\
y=0 \ \text{ on }\ \Sigma,\ y(\cdot,0)=y_0 \text{ in } \Omega,
\end{array} \right.
\end{align}
where $\zeta := (\xi,\eta)$ is a perturbation. The corresponding solution will be denoted by $y_u^\zeta$. 
In contrast with the previous section, the perturbation $\eta$ may be state and control
dependent. For this reason, here we change the notation of the set of admissible perturbations to $\hat \Gamma$.
However, Assumption \ref{Ape} will still be valid for the set $\hat \Gamma$. 
We also use the notations $C_{pe}$, $K_y$ and $R$ with the same meaning as in Subsection \ref{SmHsr}. 

In addition to Assumption \ref{Ape} we require the following that holds through the reminder of the section.

\begin{Assumption} \label{asupert1} 
The perturbation $\eta \in L^1(Q \times R)$ for every $(\xi,\eta) \in \hat \Gamma$. For a.e. $(x,t) \in Q$ the function 
$\eta(x,t,\cdot,\cdot)$ is of class $C^2$ and is convex with respect to the last argument, $u$. Moreover, 
the functions 
$\frac{\partial \eta}{\partial y}$ and $ \frac{\partial^2\eta}{\partial y^2}$ are bounded on $Q \times R$,
and the second one is continuous in $(y,u) \in R$, uniformly with respect to $(t,x) \in Q$. 
\end{Assumption}

Due to the linearity of \reff{Eyxi} and the convexity of the objective functional \reff{newequa} with respect to $u$, the proof
of the next theorem is standard.

\begin{Theorem}
For perturbations $\zeta \in \hat \Gamma$ satisfying Assumption \ref{asupert1}, 
the perturbed problem $\mbox{\rm (P$_{\zeta}$)}\xspace$ has a global solution.
\end{Theorem}

In the next two theorems, we consider sequences of problems $\{\mbox{\rm (P$_{\zeta_k}$)}\}$ with
$\zeta_k \in \hat \Gamma$. The proofs repeat the arguments in \cite[Theorem~4.2, Theorem~4.3]{CDJ2022}.

\begin{Theorem} \label{weaklimit}
Let  a sequence $\{\zeta_k \in \hat \Gamma\}_k$ converge to zero in $L^2(Q) \times L^2(Q \times R)$ and let
$u_k$ be a local solution of problem ($P_{\zeta_k}$), $k = 1, \, 2, \ldots$. 
Then any control $\bar u$ that is a weak* limit in $L^\infty(Q)$ of this sequence
is a week local minimizer in problem (P), and for the corresponding solutions
it holds that $y_{u_k} \to \bar y$ in $L^2(0,T\text{; }H^1_0(\Omega))\cap L^\infty(Q)$. 
\end{Theorem}

\begin{Theorem} \label{exsequ}
Let  $\{\zeta_k\}_k$ be as in Theorem \ref{weaklimit}.
Let $\bar u$ be a strict strong local minimizer of \Pb. Then there exists  a sequence of strong local 
minimizers $\{u_k\}$ of problems ($P_{\zeta_k}$) such that 
$u_k \stackrel{*}{\rightharpoonup} \bar u$ in $L^\infty(Q)$ and $y_{u_k}$  converges
strongly in $L^2(0,T\text{; }H^1_0(\Omega))\cap L^\infty(Q)$.
\end{Theorem}

\noindent
The next theorem is central in this section.

\begin{Theorem} \label{Tstab}
Let assumption \ref{A4}(A$_0$) be fulfilled for the reference weakly optimal control $\bar u$ in problem (P)
and the corresponding $\bar y$ and $\bar p$.
Then there exist positive numbers $\alpha $ and $C$ for which the following is fulfilled. 
For every perturbation $\zeta \in \hat\Gamma$
and for every weak local solution $u_\zeta$ of problem ($P_\zeta$) with $\| u_\zeta - \bar u \|_{L^1(Q)} \leq \alpha$,
the following estimates hold: 
\begin{enumerate}
\item If $m=0$ in \reff{genobj}:
\begin{align*}
     \|\bar u-u_\zeta\|_{L^1(Q)} & \le C\Big[ \|\xi\|_{L^2(Q)} +\| \|\frac{d}{dy}\eta \|_{L^\infty(R)} \|_{L^2(Q) }+ \|\frac{d}{du}\eta\|_{L^{\infty}(Q \times R)}\Big]^{\theta_0},
\end{align*}
\begin{align*}
   \|\bar y-y_{u_\zeta} \|_{L^2(Q)} 
      & \le C\Big[ \|\xi\|_{L^2(Q)} +\| \|\frac{d}{dy}\eta \|_{L^\infty(R)} \|_{L^2(Q) } + \|\frac{d}{du}\eta\|_{L^{\infty}(Q \times R)}\Big]^{\theta}.
\end{align*}
\item For $m\in \mathbb R$:
\begin{align*}
     \|\bar u-u_\zeta\|_{L^1(Q)}&\le C\Big[ \|\xi\|_{L^r(Q)} +\| \|\frac{d}{dy}\eta \|_{L^\infty(R)} \|_{L^r(Q) } + \|\frac{d}{du}\eta\|_{L^{\infty}(Q \times R)}\Big],
\end{align*}
\begin{align*}
   \|\bar y-y_{u_\zeta} \|_{L^2(Q)} 
      & \le C\Big[ \|\xi\|_{L^r(Q)} +\| \|\frac{d}{dy}\eta \|_{L^\infty(R)} \|_{L^r(Q) }+ \|\frac{d}{du}\eta\|_{L^{\infty}(Q \times R)}\Big]^{\theta_0}.
\end{align*}
\end{enumerate}
Here $\theta_0$ and $\theta$ are defined as in Theorem \ref{Ssr}.
\end{Theorem}

\begin{proof}
The local solution $(\bar y, \bar u)$ satisfies, together with the corresponding adjoint variable, the relations
\reff{s6}. Similarly,  $(y_{u_\zeta}, u_\zeta)$ satisfies, together with the corresponding $p_{u_\zeta}$
the perturbed optimality system \reff{s1per} with the left-hand side given by the triple
\begin{align}
\left( \begin{array}{cll}
\xi(\cdot) \\
\frac{d}{dy}(\eta(\cdot,y_{u_\zeta}(\cdot), u_\zeta(\cdot)) \\
\frac{d}{du}(\eta(\cdot,y_{u_\zeta}(\cdot), u_\zeta(\cdot)).
\end{array} \right)
\end{align}
Since it is assumed that $\| u_\zeta - \bar u \|_{L^1(Q)} \leq \alpha$ 
we may apply Theorem \ref{Ssr}
(here we choose the same $\alpha$ as in this theorem) to prove the inequalities in the theorem.
\end{proof}

The proof of theorems \ref{coro1} and \ref{coro2} follows in the same spirit but using Theorem \ref{SsHr} 
instead of Theorem \ref{Ssr}.
We make an additional assumption for the perturbation $\eta$ in the objective functional, namely, that 
$\rho := \frac{d}{du}(\eta(\cdot,y_{u_\zeta}(\cdot), u_\zeta(\cdot))$ satisfies \eqref{perturbeq}, i.e.
\begin{equation}
\frac{d}{du}(\eta(\cdot,y_{u_\zeta}(\cdot), u_\zeta(\cdot))\in D(\mathcal L^*).
\label{crp}
\end{equation}
For an explanation of the condition \eqref{crp}, we refer to the proof of Theorem \ref{SsHr}.

\begin{Theorem}
Let $m=0$ and Assumption \ref{A4}(A$_1$) be fulfilled for the reference strongly optimal control $\bar u$ in problem (P).
Then there exist positive numbers $\alpha $ and $C$ for which the following is fulfilled. For every perturbation 
$\zeta \in \hat\Gamma$
and for every local solution $u_\zeta$ of problem ($P_\zeta$) with $\| y_{u_\zeta} - \bar y \|_{L^\infty(Q)} \leq \alpha$,
the following estimates hold.
\begin{align*}
&\|\bar u-u_\zeta \|_{L^1(Q)}\le C\Big(\|\mathcal L^* \frac{d}{du}(\eta(\cdot,y_{u_\zeta}(\cdot), u_\zeta(\cdot)) \|_{L^{2}(Q)}+\|\xi\| _{L^2(Q)}+\| \|\frac{d}{dy}\eta \|_{L^\infty(R)} \|_{L^2(Q) }\Big)
\end{align*}
and all together
\begin{align*}
&\| \bar y-y_{u_\zeta}\|_{L^2(Q)}\le C\Big(\|\mathcal L^* \frac{d}{du}(\eta(\cdot,y_{u_\zeta}(\cdot), u_\zeta(\cdot))\| _{L^{2}(Q)}+\| \xi\| _{L^2(Q)}+\| \|\frac{d}{dy}\eta \|_{L^\infty(R)} \|_{L^2(Q) }\Big)^{\theta_0},
\end{align*}
where $\theta_0$ is defined in Theorem \ref{Ssr}.
\label{coro1}
\end{Theorem}

\begin{Theorem}
Let $m=0$ and let Assumption \ref{A4}(A$_2$) be fulfilled for the reference strongly optimal control $\bar u$ in problem (P).
Then there exist positive numbers $\alpha $ and $C$ for which the following is fulfilled. For every perturbation 
$\zeta \in \hat\Gamma$
and for every local solution $u_\zeta$ of problem ($P_\zeta$) with $\| y_{u_\zeta} - \bar u \|_{L^\infty(Q)} \leq \alpha$,
the following estimates hold: 
\begin{align*}
&\|\bar y-y_{u_\zeta}\|_{L^2(Q)}\le C\Big(\|\mathcal L^* \frac{d}{du}(\eta(\cdot,y_{u_\zeta}(\cdot), u_\zeta(\cdot)) \|_{L^{2}(Q)}+\|\xi\|_{L^2(Q)}+\| \|\frac{d}{dy}\eta \|_{L^\infty(R)} \|_{L^2(Q) }\Big).
\end{align*}
\label{coro2}
\end{Theorem}

\begin{Remark}
The constraint that $u_\zeta$ needs to be close to the reference solution $\bar u$ in the theorems above is not a big restriction. 
This is clear, since Assumption \ref{A4} implies that $\bar u$ satisfies \eqref{E3.12}. Hence, $\bar u$ is a strict strong local 
minimizer of \Pb and, consequently, Theorem \ref{exsequ} ensures the existence of a family $\{u_{\zeta_k}\}$, 
${\zeta_k \in \hat \Gamma}$,
of strong local minimizers of problems \mbox{\rm (P$_{\zeta}$)}\xspace satisfying the conditions of Theorem \ref{Ssr} 
or \ref{SsHr}. 
\label{R4.1}
\end{Remark}

\begin{Example}[Tikhonov regularization]
We consider the optimal control problem
\[
\mbox{\rm (P$_{\lambda}$)}\xspace \  \min_{u \in \Uad} J_\lambda(u) := \int_Q L(x,t,y(x,t),u(x,t)) + 
\frac{\lambda}{2}\int_Q u(x,t)^2\dx\dt,
\]
subject to \reff{see1} and \reff{const}.
As before, $\bar u$ denotes a strict strong solution of problem (P)$\equiv\mbox{\rm (P$_{0}$)}\xspace$. 
We assume that $\bar u$ satisfies Assumption \ref{A4}$(A_0)$. From Theorem \ref{exsequ} we know that
for every sequence $\lambda_k>0$ converging to zero
there exists a sequence of strong local minimizer $\{ u_{\lambda_k}\}_{k=1}^\infty$ 
such that $u_{k}\to \bar u$ in $L^1(Q)$ for 
$k\to \infty$, thus for a sufficiently large $k_0$ we have that for all $k>k_0$
\begin{align*}
\|y_{\bar u}-y_{u_k}\|_{L^2(Q)}+\| p_{\bar u}-p_{u_k} \|   _{L^2(Q)}\le 
C\Big(\lambda_k \Big)^{\theta},
\end{align*}
\begin{equation*}
\|\bar u-u_{k}\|_{L^1(Q)}\le C\lambda_k,
\end{equation*}
where $\theta $ is defined in Theorem \ref{Ssr}.
\end{Example}


\section{Examples}

Here we present two examples that show particular applications in which different assumptions are involved. 

\begin{Example}[Negative curvature] \label{Ex2}
We begin with an optimal control problem, that has negative curvature.
The parabolic equation has the form
\begin{equation} \label{Eye}
\left\{ \begin{array}{lll}
\frac{dy}{dt}+\mathcal A y+\exp(y)& =u\  &\text{ in }\ Q,\\
y=0 \text{ on } \Sigma,\quad y(\cdot,0)& =  y_0\ &\text{ on } \Omega.\\
\end{array} \right.
\end{equation}
Let $0 \leq g \in L^2(Q)$ be a function satisfying the structural assumption \eqref{struct}. 
We consider the optimal control problem  
\begin{equation*}
        \min_{u\in \mathcal U} \Big \{ J(u):=\int_Q  (y_u +gu) \dx  \dt \Big \}
\end{equation*}
subject to \reff{Eye} and with control constraints 
\begin{equation}
		\Uad := \{u \in L^\infty(Q) \vert \ 0\leq u_a \le u \le u_b\ \text{ for a.a. } (x,t) \in Q\}.
	\end{equation}
By the weak maximum principle $y_{u_a}-y_u\leq0$ for all $u\in \mathcal U$ and $\bar u:=u_a$ constitutes an optimal solution. Further, by the weak maximum principle, the adjoint-state $\bar p$ and the linearized states 
$z_{\bar u,u-\bar u}$ for all $u \in \mathcal U$, are non-negative.
Moreover, we have
\begin{equation*}
   J'(\bar u)(u-\bar u) = \int_Q(\bar p+g)(u - \bar u)\dx\dt \ge 0, 
\end{equation*}
\begin{equation*}
J''(\bar u)(u-\bar u)^2=\int_Q w_{\bar u,u-\bar u}\dx \dt=\int_Q -\bar p \exp (\bar y) z_{\bar u,u-\bar u}^2\dx\dt< 
    0,
\end{equation*}
for all $u \in \Uad$. Since $g$ satisfies the structural assumption, there exists a constant $C>0$ such that
\[
\int_Q g(u-\bar u)\dx \dt\geq C\| u-\bar u\|_{L^1(Q)}^2\ \ \forall u\in \mathcal U.
\]
On the other hand, integrating by parts we obtain 
\begin{equation}
\int_Q \bar p(u-\bar u)\dx \dt=\int_Q z_{\bar u, u-\bar u}\dx \dt.
\label{abs}
\end{equation}
If  for $u\in \mathcal U$ with $\| u-\bar u\|_{L^1(Q)}$ or $\|y_u-\bar y\|_{L^\infty(Q)}$ sufficiently small such that 
\[
\frac{1}{2\|\bar p \exp(\bar y)\|_{L^\infty(Q)}}>\| z_{\bar u,u-\bar u}\|_{L^\infty(Q)},
\]
we can absorb the term $J''(\bar u)(u-\bar u)^2$ by estimating
\begin{align}
\int_Q \bar p(u-\bar u)\dx \dt+J''(\bar u)(u-\bar u)&=
\int_Q z_{\bar u, u-\bar u}(1-\bar p \exp (\bar y) z_{\bar u,u-\bar u})\dx\dt\\
&\geq \frac{1}{2}\int_Q z_{\bar u, u-\bar u}\dx\dt\geq \frac{K}{2}\|z_{\bar u, u-\bar u}\|_{L^2(Q)}^2,
\end{align}
where the last inequality is a consequence of the boundedness of $\mathcal U\subset L^\infty(Q)$ that implies the existence of a constant $K>0$ such that $\|z_{\bar u, u-\bar u}\|_{L^1(Q)}\geq K\|z_{\bar u, u-\bar u}\|_{L^2(Q)}^2$ for all $u\in \mathcal U$.
Altogether, we find 
\begin{align*}
J'(\bar u)(u-\bar u)+J''(\bar u)(u-\bar u)^2 &\geq C \| u-\bar u\|_{L^1(Q)}^2 +  \frac{K}{2} \| z_{\bar u, u-\bar u}\|_{L^2(Q)}^2 \\
& \geq \sqrt{\frac{CK}{2}} \| u-\bar u\|_{L^1(Q)}\| z_{\bar u, u-\bar u}\|_{L^2(Q)} \ \ \forall u\in \mathcal U.
\end{align*}
Thus, Assumption \ref{A4}$(A_1)$ is fulfilled and we can apply Theorem \ref{SsHr} to obtain a stability result.
\end{Example}

\begin{Example}[State stability]
We consider a tracking type objective functional where the control does not appear explicitly and for which
we will verify  $(A_2)$.
As perturbations we consider functions 
$\zeta =(\xi,\eta, \rho) \in D(\mathcal L^*)\times L^r(Q)\times L^r(Q) \times D(\mathcal L^*)$. 
Denote by $y_d$ the solution of this equation with $u = u_a$ and consider the problem
\[
   \min_{u \in \Uad} \Big\{J(u) :=  \frac{1}{2}\int_Q (y(x,t)-y_d(x,t))^2\dx\dt+\int_Q\eta y\dx\dt+\int_Q\rho u\dx\dt \Big\},
\]
subject to the same constraints as inn Example \ref{Ex2}.
For a local minimizer $\bar u$ of the unperturbed problem ($\zeta=0$), it holds 
\begin{align*}
J'(\bar u)(u-\bar u)&=\int_Q(\bar y(x,t)-y_d(x,t))z_{\bar u, u-\bar u}\ \text{dxdt}\geq 0 \quad \forall u \in \Uad,\\
J''(\bar u)(u-\bar u)&=\int_Q(\bar y(x,t)-y_d(x,t))w_{\bar u, u-\bar u}+z_{\bar u, u-\bar u}^2 \ \text{dxdt}\\
&=\int_Q(1-\bar p \,exp(\bar y ))z_{\bar u, u-\bar u}^2 \ \text{dxdt} \quad \forall u \in \Uad,
\end{align*}
where  $p$ solves
\begin{equation*}
\left\{\begin{array}{lll} 
-\frac{d \bar p}{dt}+\mathcal A^*\bar p + \exp(\bar y)\bar p = \bar y-y_d& \text{ in } Q,\\ 
\bar p = 0\text{ on } \Sigma, \ p(\cdot,T)=0\ &\text{ on } \Omega.\\
\end{array}\right.
\end{equation*}
If the optimal state tracks $y_d$ such that $\| \bar y-y_d\|_{L^r(Q)} \leq\frac{1}{2C_r \| \exp(\bar y)\|_{L^\infty(Q)}}$ we find that 
$(A_2)$ holds. From Theorem \ref{coro1} we obtain the existence of a constant $\kappa>0$ such that
\begin{equation*}
     \| y_{\bar u}-y_{\zeta}\|_{L^2(Q)}+\|  p_{\bar u}-p_\zeta\|_{L^2(Q)}
     \le \kappa\Big( \|\xi\|_{L^2(Q)} + \| \eta\|_{L^2(Q)} + \| \mathcal L^*\rho \|_{L^{2}(Q)} \Big),
\end{equation*}
for every perturbation $\zeta \in \hat\Gamma$
and for every local solution $u_\zeta$ of problem (P) with $\| y_{u_\zeta} - \bar u \|_{L^\infty(Q)} \leq \alpha$.
\end{Example}
\appendix
\section{Appendix} \label{secA1}

\begin{Lemma}
Suppose $r >1+ \frac{n}{2}$ and $s\in [1,\frac{n+2}{n})\cap[1,2]$. The following statement is fulfilled  for all $ u,\bar u \in \Uad$.
There exist positive constants $K_r$, $M_s$ and $N_{r,s}$ depending on $s$ and $r$ such that
\begin{align}
	&\|   y_u - y_{\bar u} - z_{\bar u,u - \bar u}\|   _{ C(\bar Q)}\le K_r\|   y_u - y_{\bar u}\|^2_{L^{2r}(Q)},\label{E2.11}\\
	&\|   y_u - y_{\bar u} - z_{\bar u,u - \bar u}\|   _{L^{s}(Q)}
	\le M_s \|   y_u-y_{\bar u}\|^{2-s}_{C(\bar Q)}\|   y_u-y_{\bar u}\|^s_{L^{s}(Q)},\label{E2.12}\\
	&\|   y_u - y_{\bar u} - z_{\bar u,u - \bar u}\|   _{L^{2}(Q)}
	\le N_{r,s}\|y_u-y_{\bar u}\|^{2-\frac{s^2}{2}}_{C(\bar Q)}\|y_u-y_{\bar u}\|_{L^{s}(Q)}^{\frac{s}{2}}.\label{E2.13L2}
\end{align}
\label{spreprep}
\end{Lemma}

\begin{proof}
Let us denote $\phi:= y_u - y_{\bar u} - z_{\bar u,u - \bar u} \in W(0,T) \cap C(\bar Q)$. From the equations satisfied by the three functions and by the mean value theorem $\phi$ satisfies
\[
\frac{d \phi}{dt}+\mathcal{A}\phi + \frac{\partial f}{\partial y}(x,t,y_{\bar u})\phi =  \Big[\frac{\partial f}{\partial y}(x,t,y_{\bar u}) - \frac{\partial f}{\partial y}(x,t,y_\theta)\Big](y_u - y_{\bar u}),
\]
where $y_\theta(x,t) = y_{\bar u}(x,t) + \theta(x,t)(y_u(x,t) - y_{\bar u}(x,t))$ with $\theta: Q \longrightarrow [0,1]$ measurable. Applying again the mean value theorem we obtain
\[
\frac{d \phi}{dt}+\mathcal{A}\phi + \frac{\partial f}{\partial y}(x,t,y_{\bar u})\phi =  \frac{\partial^2f}{\partial y^2}(x,t,y_\vartheta)(y_u - y_{\bar u})^2
\]
with $y_\vartheta(x,t) = y_{\bar u}(x,t) + \vartheta(x,t)(y_\theta(x,t) - y_{\bar u}(x,t))$ and $\vartheta:Q \longrightarrow [0,1]$ measurable. By Theorem \ref{mainex} and Remark \ref{Rlip} we infer the existence of constants $C_r, \bar C$ independent of $u, \bar u \in \Uad$ and $\frac{\partial f}{\partial y}(x,t,y_{\bar u})$ such that
\[
\|   \phi\|   _{C(\bar Q)} \le C_r\bar{C}\|   (y_{u}-y_{\bar u})^2\|   _{L^r(Q)} = C_r\bar{C}\|   y_{u}-y_{\bar u}\|   ^2_{L^{2r}(Q)},
\]
which proves \eqref{E2.11} with $K_r:= C_r\bar{C}$. To prove \eqref{E2.12}, we use Lemma \ref{estLs}, Remark \ref{Rlip} and \eqref{unibound} to obtain
\begin{align}
\|\phi\|  _{L^{s}(Q)}&\le C_{s'}\bar{C}\|(y_u-y_{\bar u})^2\|_{L^1(Q)}=C_{s'}\bar{C}\|   y_u-y_{\bar u}\|   ^2_{L^2(Q)}\le C_{s'}\bar{C}\|y_u-y_{\bar u}\|^{2-s}_{C(\bar Q)}\|   y_u-y_{\bar u}\|^s_{L^s(Q)}.
\label{LsLs}
\end{align}
Taking $M_s:=  C_{s'}\bar{C}$, \eqref{E2.12} follows. The inequality, \eqref{E2.13L2}, follows from \eqref{E2.12} and \eqref{E2.11} of Lemma \ref{spreprep} by estimating
\begin{equation*}
\begin{aligned}
\| \phi\| _{L^{2}(Q)}\le \|   \phi\|   _{C(\bar Q)}^{\frac{2-s}{2}}\|\phi\|   _{L^s(Q)}^{\frac{s}{2}}&\le K_{r}^{\frac{2-s}{2}}\|y_u-y_{\bar u}\|_{L^{2r}(Q)}^{\frac{2(2-s)}{2}}\bigg[M_s^{\frac{s}{2}}\|   y_u-y_{\bar u} \|   _{C(\bar Q)}^{\frac{(2-s)s}{2}}\|   y_u-y_{\bar u}\|   _{L^s(Q)}^{\frac{s^2}{2}}\bigg]\\
&\le K_{r}^{\frac{(2-s)}{2}}M_s^{\frac{s}{2}}\vert  Q\vert  ^{\frac{2-s}{2r}}\| y_u-y_{\bar u} \| _{C(\bar Q)}^{2-s+\frac{(2-s)s}{2}}\|   y_u-y_{\bar u}\|   _{L^s(Q)}^{\frac{s^2}{2}}.
\end{aligned}
\end{equation*}
Defining $N_{r,s}:=K_{r}^{\frac{(2-s)}{2}}M_s^{\frac{s}{2}}\vert  Q\vert  ^{\frac{2-s}{2r}}$ and noticing that $2-s+\frac{(2-s)s}{2}=2-\frac{s^2}{2}$ proves the claim.
\end{proof}

\begin{proof}{\em of Proposition \ref{sprep}.}
We prove \eqref{E33} by applying Theorem \ref{uniqueexlin} to 
$\psi:=z_{\bar u,v}-z_{u_\theta,v}$, that solves
\begin{align}
\frac{d \psi}{dt}+\mathcal{A}\psi + \frac{\partial f}{\partial y}(x,t,y_{\bar u})\psi &= \Big[ \frac{\partial f}{\partial y}(x,t,y_{u_\theta})-\frac{\partial f}{\partial y}(x,t,y_{\bar u})\Big]z_{u_\theta,v} = \frac{\partial^2f}{\partial y^2}(x,t,y_\vartheta)(y_{\bar u} - y_{u_\theta})z_{u_\theta,v}\label{last}.
\end{align}
To prove (\ref{E2.14.2}), we use \eqref{E2.13L2} with $s=\sqrt{2}$  to estimate
\begin{align*}
&\|y_u-y_{\bar u} \|_{L^2(Q)}\leq \| \phi\|_{L^2(Q)}+\|z_{\bar{u}, u-\bar u}\|   _{L^2(Q)}\leq N_{r,\sqrt{2}} \|y_u-y_{\bar u}\|_{C(\bar Q)}\|   y_u-y_{\bar u}\|_{L^{\sqrt{2}}(Q)}+\| z_{\bar{u}, u-\bar u}\| _{L^2(Q)}.
\end{align*}
Using fact that by the H\"older inequality $\| y_u-y_{\bar u} \|_{L^{\sqrt{2}}(Q)}\leq \vert  Q\vert^{\frac{1}{\sqrt{2}}-\frac{1}{2}}\|y_u-y_{\bar u}\| _{L^2(Q)}$, the claim follows.
For the other direction, we select again $s=\sqrt{2}$ in \eqref{E2.13L2} and find
\begin{align*}
\|z_{\bar{u}, u-\bar u}\|_{L^2(Q)}&\leq \|\phi\| _{L^2(Q)}+\| y_u-y_{\bar u}\|_{L^2(Q)}\\
&\leq  N_{r,\sqrt{2}} \| y_u-y_{\bar u}\|_{C(\bar Q)}\| y_u-y_{\bar u}\|_{L^{\sqrt{2}}(Q)}+\|y_u-y_{\bar u}\|_{L^2(Q)}\\
&\leq\bigg( N_{r,\sqrt{2}}\vert  Q\vert^{\frac{1}{\sqrt{2}}-\frac{1}{2}}\|y_u-y_{\bar u}\|_{C(\bar Q)}+1\bigg) \|y_u-y_{\bar u} \|_{L^2(Q)}.
\end{align*}
Finally, for \eqref{E2.15.2} we use \eqref{E33} and estimate 
\begin{equation*}
\begin{aligned}
\|   z_{\bar u,v}\|   _{L^2(Q)}&\leq \|   z_{\bar u,v}-z_{u,v} \|   _{L^2(Q)}+\|    z_{ u,v}\|   _{L^2(Q)}\le K_2\sqrt[2]{\vert  Q\vert  }\|   y_u - y_{\bar u}\| _{C(\bar Q)}\|   z_{\bar u,v}\|   _{L^2(Q)}+\|    z_{ u,v}\|   _{L^2(Q)}.
\end{aligned}
\end{equation*}
Choosing $\varepsilon = [2K_2\sqrt[2]{\vert  Q\vert  }]^{-1}$ proves the first part. The second inequality follows in a similar way. The estimates with respect to the $\|\cdot\|_{L^\infty(Q)}$ follow by similar reasoning, using \eqref{E2.11}.
\end{proof}

\begin{proof}{\em of Proposition \ref{equivasu}.}
Let us prove first the implication $(A_k) {\Rightarrow} (B_k)$ for any $k\in\{0,1,2\}$.
Given $u \in \Uad$,  by the mean value theorem
\[
\frac{d(y_u-\bar y)}{dt}+\mathcal{A}(y_u - \bar y) + \frac{\partial f}{\partial y}(x,\bar y + \theta(y_u - \bar y))(y_u - \bar y) = u - \bar u.
\]
Using \eqref{clr} in Theorem \ref{mainex} we obtain that
\[
\|y_u - \bar y\|_{C(\bar Q)} \le C_r\|u - \bar u\|_{L^r(Q)} \le C_r(2M_{\mathcal U})^{\frac{r-1}{r}}\|u - \bar u\|^{\frac{1}{r}}_{L^1(Q)}.
\]
Then, by $\tilde\alpha_k:= \frac{\alpha_k ^r}{C^r_r(2M_{\mathcal U})^{r-1}}$, we obtain that $(A_k)$ implies $(B_k)$ with 
$\gamma_k = \tilde \gamma_k$.

To prove the converse implication, $(B_k) {\Rightarrow} (A_k)$, we assume that ($B_k$) holds, but ($A_k$) fails. 
Then for every integer $l \ge 1$ there exists an element 
$u_l \in \Uad$ such that
\begin{equation}\label{E5.5}
\begin{aligned}
J'(\bar u)(u_l - \bar u) &+ J''(\bar u)(u_l - \bar u)^2< \frac{1}{l}\|u_l - \bar u\|^{2-k}_{L^1(Q)}\|   z_{\bar u,u_l - \bar u}\|^{k}_{L^2(Q)}\  \text{ and } \ \|   y_{u_l} - \bar y\|   _{C(\bar Q)} 
     < \frac{1}{l}.
\end{aligned}
\end{equation}
Since $\{u_l\}_{l = 1}^\infty \subset \Uad$ is bounded in $L^\infty(Q)$, we can extract a subsequence, 
denoted in the same way, such that $u_l \stackrel{*}{\rightharpoonup} u$ in $L^\infty(Q)$. On one side, 
\eqref{E5.5} implies that $y_{u_l} \to \bar y$ in $L^\infty(Q)$. On the other side, $u_l \stackrel{*}{\rightharpoonup} u$ 
in $L^\infty(Q)$ implies weak convergence in $L^r(Q)$. From \eqref{semilinweak}, the convergence 
$y_{u_l} \to y_u$ in $L^\infty(Q)$ follows. Then, $y_u = \bar y$ and, consequently, $u = \bar u$ holds.  
But Assumption($B_0$)
implies that $\bar u$ is bang-bang, and hence the weak convergence $u_l \stackrel{*}{\rightharpoonup} \bar u$  in 
$L^\infty(Q)$ yields the strong convergence $u_l \to \bar u$ in $L^1(Q)$; see \cite[Proposition 4.1 and Lemma 4.2]{DJV2022}. 
Then, for $k=0$, \eqref{E5.5} contradicts $(B_0)$. The same argument holds for $(B_1)$ and $(B_2)$ under the additional 
condition that $\bar u $ is bang-bang and noticing that $\|z_{\bar u,u_l-\bar u}\|_{C(\bar Q)} \leq 3/2\|y_{u_l}-\bar y\|_{C(\bar Q)}$ 
by Lemma \ref{sprep}.

\end{proof}

A proof of the following Lemma can be found in \cite[Lemma 3.5]{CDJ2022} or \cite[Lemma 3.5]{CMR}.
\begin{Lemma}
Given $\bar u \in \Uad$ with associated state $\bar y$. Then, the following estimate holds
\begin{equation}
\|y_{\bar u + \theta(u - \bar u)} - \bar y\|_{C(\bar Q)} \le B\|y_u - \bar y\|_{C(\bar Q)} \quad \forall \theta \in [0,1]\ \text{ and }\ \forall u \in \Uad,
\label{E3.14}
\end{equation}
where $B:= (2C_r\bar{C}\sqrt[r]{\vert  Q\vert  }M_{\mathcal U} + 1)$, $C_r$ is the constant of Lemma \ref{C-est} and $\bar{C}$ is the one from Remark \ref{Rlip}.
\label{L3.1}
\end{Lemma}
We proof the analogous statement for the adjoint-state. For an elliptic state equation, it was also done in \cite[Lemma 3.7]{CDJ2022}.
\begin{Lemma}
Given $\bar u \in \Uad$ with associated state $\bar y$ and adjoint-state $\bar p$, there exists a constant $\tilde B>0$ such that
\begin{equation}
\|p_{\bar u + \theta(u - \bar u)} - \bar p\|_{C(\bar Q)} \le \tilde B(\|y_u - \bar y\|_{C(\bar Q)} +\vert m\vert \| u-\bar u\|_{L^1(Q)}^{\frac{1}{r}}),
\label{E3.16}
\end{equation}
for all $\theta \in [0,1] \text{ and }u \in \Uad$.
\label{L3.2}
\end{Lemma}

\begin{proof}
Let us prove \eqref{E3.16}. Given $u \in \mathcal U$ and $\theta \in [0,1]$, let us denote $u_\theta = \bar u + \theta(u - \bar u)$, $y_\theta = y_{u_\theta}$, and $p_\theta = p_{u_\theta}$. Subtracting the equations satisfied by $p_\theta$ and $\bar p$ we get with the mean value theorem
\begin{align*}
&-\frac{d}{dt}(p_\theta - \bar p)+\mathcal A^ *(p_\theta - \bar p) + \frac{\partial f}{\partial y}(x,t,\bar y)(p_\theta - \bar p) = \frac{\partial L}{\partial y}(x,t,y_\theta,u_\theta) - \frac{\partial L}{\partial y}(x,t,\bar y,\bar u)\\ 
&+ \Big[\frac{\partial f}{\partial y}(x,t,\bar y) - \frac{\partial f}{\partial y}(x,t,y_\theta)\Big]p_\theta\\
& = \Big[\frac{\partial^2L}{\partial y^2}(x,t,y_\vartheta) -p_\theta\frac{\partial^2f}{\partial y^2}(x,t,y_\vartheta)\Big](y_\theta - \bar y)+m(u_\theta -\bar u),
\end{align*}
where $y_\vartheta = \bar y + \vartheta(y_\theta - \bar y)$ for some measurable function $\vartheta:Q \longrightarrow [0,1]$. Now, we can apply again Theorem \ref{mainex} and Remark \ref{Rlip} to conclude from the above equation
\begin{align*}
\|p_\theta - \bar p\| _{C(\bar Q)}&\le C_r(\bar{C} + M_{\mathcal U} \bar{C})\sqrt[r]{\vert  Q\vert}\|y_\theta - \bar y\|_{C(\bar Q)}+\vert m\vert \theta C_r\|u-\bar u\|_{L^r(Q)}\\
&\leq \tilde B\|y_u - \bar y\|_{C(\bar Q)}+\vert m\vert \|u-\bar u\|_{L^1(Q)},
\end{align*}
where $\tilde B:=C_r((\bar{C} + M_{\mathcal U} \bar{C})\sqrt[r]{\vert  Q\vert}B+(2M_\mathcal U)^{\frac{r-1}{r}} )$, with $B$ being the constant from Lemma \ref{L3.1}.
Then, \eqref{E3.16} follows by applying Lemma \ref{L3.1}.
\end{proof}

\begin{proof}{\em of Lemma \ref{biglemmat}.}
The second variation of the objective functional is given by Theorem \ref{T3.1}. Let us denote $u_\theta$, $y_\theta$, and $\varphi_\theta$ as in the proof of Lemma \ref{L3.2}. From \eqref{E3.5} we obtain
\begin{align*}
&\vert  [J''(\bar u + \theta(u - \bar u)) - J''(\bar u)](u - \bar u)^2\vert\\
&\le \int_Q\Big\vert  \Big[\frac{\partial^2L_0}{\partial y^2}(x,t,y_\theta) - \frac{\partial^2L_0}{\partial y^2}(x,t,\bar y)\Big] z_{u_\theta,u - \bar u}^2\Big\vert   \dx\dt+ \int_Q\Big\vert  (\bar\varphi - \varphi_\theta)\frac{\partial^2f}{\partial y^2}(x,t,y_\theta)z_{u_\theta,u - \bar u}^2\Big\vert  \dx\dt\\
&+ \int_Q\Big\vert  \bar\varphi\Big[\frac{\partial^2f}{\partial y^2}(x,t,\bar y) - \frac{\partial^2f}{\partial y^2}(x,t,y_\theta)\Big]z_{u_\theta,u - \bar u}^2\Big\vert  \dx\dt\\
& + \int_Q\Big\vert  \Big[\frac{\partial^2L_0}{\partial y^2}(x,t,\bar y)- \bar\varphi\frac{\partial^2f}{\partial y^2}(x,t,\bar y)\Big](z^2_{u_\theta,u - \bar u} - z^2_{\bar u,u - \bar u})\Big\vert  \dx\dt+2\Big\vert  \int_Q (u-\bar u)m\Big[ z_{u_\theta,u - \bar u}-z_{\bar u,u - \bar u}\Big]\dx\dt\Big\vert\\
& = I_1 + I_2 + I_3 + I_4+I_5.
\end{align*}
We consider the case $m=0$ first. Let us consider the terms $I_i$, $i\in \{1,..,4\}$.
For $I_1$, we deduce from Remark \ref{Rlip}, \eqref{E3.14}, and \eqref{E2.15.2} that for every $\rho_1 > 0$ there exists $\varepsilon_1 > 0$ such that
\[
I_1 \le \rho_1\|z_{\bar u,u - \bar u}\|^2_{L^2(Q)}\quad  \text{if}\quad \|y_u - \bar y\|_{C(\bar Q)} < \varepsilon_1.
\]
To deal with $I_2$, we use Remark \ref{Rlip}, \eqref{E2.15.2}, and \eqref{E3.16} to obtain for every $\rho_2 > 0$ the existence of a $\varepsilon_2 > 0$ such that
\[
I_2 \le \rho_2\|z_{\bar u,u - \bar u}\|^2_{L^2(Q)}\quad \text{if}\quad \|y_u - \bar y\|_{C(\bar Q)} < \varepsilon_2.
\]
The estimate for $I_3$ follows from \eqref{E2.15.2} and Remark \ref{Rlip}. Thus for every $\rho_3>0$, there exists $\varepsilon_3>0$ with
\[
I_3 \le \rho_3\|z_{\bar u,u - \bar u}\|^2_{L^2(Q)} \quad \text{if} \quad\|y_u - \bar y\|_{C(\bar Q)} < \varepsilon_3.
\]
For $I_4$ we infer by Remark \ref{Rlip}, \eqref{E2.13L2}, \eqref{E2.15.2} and \eqref{E3.14} that for every $\rho_4 > 0$ there exists $\varepsilon_4> 0$ such that
\begin{align*}
I_4 &\le (\bar{C} + M_{\mathcal U}\bar{C})\|   z_{u_\theta,u - \bar u} + z_{\bar u,u - \bar u}\|_{L^2(Q)}\| z_{u_\theta,u - \bar u} - z_{\bar u,u - \bar u}\|   _{L^2(Q)}\\
& \le \frac{C_{2}5}{2}(\bar{C} + M_{\mathcal U}\bar{C})\| z_{\bar u,u - \bar u}\|   _{L^2(Q)}\|   y_\theta - \bar y\| _{C(\bar Q)}\|   z_{\bar u,u - \bar u}\|   _{L^2(Q)}\\
&\le \rho_4\| z_{\bar u,u - \bar u}\|^2_{L^2(Q)} \quad \text{if}\quad \|y_u - \bar y\|   _{C(\bar Q)} < \varepsilon_4.
\end{align*}
Taking $\rho_i$ small enough such that $I_i < \frac{\rho}{4}$ for every $i\in \{1,..,4\}$ and setting $\varepsilon = \min_{1 \le i \le 4}\varepsilon_i$, the first claim follows. \\
For the case $m\neq 0$, we need to additionally estimate $I_5$ and reconsider the term $I_2$. We recall that for the case $m\neq 0$, we assume that $\| u-\bar u\|_{L^1(Q)}$ is sufficiently small.
To estimate $I_5$ we use that $z_{\bar u,v}$ satisfies equation \eqref{stddt} and that $\psi:=z_{\bar u,u - \bar u}-z_{u_\theta,u - \bar u}$ solves
\begin{align}
\frac{d \psi}{dt}+\mathcal{A}\psi + \frac{\partial f}{\partial y}(x,t,y_{\bar u})\psi &= \Big[ \frac{\partial f}{\partial y}(x,t,y_{u_\theta})-\frac{\partial f}{\partial y}(x,t,y_{\bar u})\Big]z_{u_\theta,u - \bar u}= \frac{\partial^2f}{\partial y^2}(x,t,y_\vartheta)(y_{\bar u} - y_{u_\theta})z_{u_\theta,u - \bar u}\label{theta2},
\end{align}
where we used the mean value theorem to infer the existence of a function $\vartheta$ such that \eqref{theta2} holds.
We use Remark \ref{Rlip}, \eqref{E2.15.2}, Lemma \ref{estLs} and \eqref {E3.14} to estimate
\begin{align*}
&2 \Big\vert  \int_Q (u-\bar u)m\Big[ z_{u_\theta,u - \bar u}-z_{\bar u,u - \bar u}\Big]\dx\dt \Big\vert \leq2\vert  m\vert \| u - \bar u\|_{L^{s'}(Q)}\|  z_{u_\theta,u - \bar u}-z_{\bar u,u - \bar u}\|_{L^s(Q)}\\
&\leq 2\vert m\vert (2M_{\mathcal U})^{\frac{s'-1}{s'}} \| u - \bar u\|_{L^{1}(Q)}^{\frac{1}{s'}}\|  z_{u_\theta,u - \bar u}-z_{\bar u,u - \bar u}\|_{L^s(Q)}\\
&\le \vert m\vert \bar{C}C_{s'}B(2M_{\mathcal U})^{\frac{s'-1}{s'}}\| u - \bar u\|_{L^{1}(Q)}^{\frac{1}{s'}} \|y_{u_\theta} - \bar y\|    _{L^2(Q)}\| z_{\bar u_{\theta},u - \bar u}\| _{L^2(Q)}\\
&\leq \rho_5\| z_{\bar u,u - \bar u}\| _{L^2(Q)}^2 \quad \text{if}\quad \|u - \bar u\|_{L^{1}(Q)} < \varepsilon_5.
\end{align*}
We remark, that to make the last step, we used that \eqref {E3.14} holds also if the $\|\cdot\|_{L^\infty(Q)}$-norm is exchanged with the $\|\cdot\|_{L^2(Q)}$-norm. This can be seen in the proof of \cite[Lemma 3.5]{CDJ2022}.
The validity of the estimates for $I_i$ for $i\in\{1,3,4\}$ holds, noticing that by \eqref{clr}, $\|u-\bar u\|_{L^1(Q)}<\frac{\varepsilon^r}{C_r^r(2M_\mathcal U)^\frac{r-1}{2r}}$, implies $\|y_u - \bar y\|_{C(\bar Q)} < \varepsilon$. For the term $I_2$ we use Remark \ref{Rlip}, \eqref{E2.15.2}, and \eqref{E3.16}, to find for any $\rho_2>0$ a $\varepsilon_2>0$ such that
\begin{align}
I_2 &\le\frac{9}{4}\bar{C}\tilde B(C_r(2M_{\mathcal U})^\frac{r-1}{r}+\vert m\vert )\| u-\bar u\|_{L^1(Q)}^{\frac{1}{r}} \| z_{\bar u,u - \bar u}\|_{L^2(Q)}^2\leq \rho_2\|z_{\bar u,u - \bar u}\|^2_{L^2(Q)}\quad \text{if}\quad \| u - \bar u\|   _{L^1(Q)} < \varepsilon_2.
\end{align}
Taking $\varepsilon:= \min_{1 \le i \le 5}\varepsilon_i$, completes the proof.
\end{proof}

\begin{proof}{\em of Corollary \ref{bigcorollary}.}
Let $s\in[1,\frac{n+2}{n})\cap[1,2]$. We first consider the case $m=0$. Using that $L_0$ and $f$ satisfy the assumption in Remark \ref{Rlip} and arguing as in the proof of 
Lemma \ref{biglemmat}, there exists $\varepsilon>0$  and a constant $P>0$ such that
\[
\vert  [J''(\bar u + \theta(u - \bar u)) - J''(\bar u)](u-\bar
u)^2\vert   <P \|   y_u-y_{\bar u}\| _{L^\infty(Q)}\|   z_{\bar u,u-\bar
u}\| _{L^2(Q)}^2
\] 
for all $u\in \mathcal U$ with $\| y_u-y_{\bar u}\|_{L^\infty(Q)}<\varepsilon$.
To prove \eqref{E3.17.1}, we select $l_1,l_2\geq 0$ with $l_1+l_2=1$ and use the estimate
\begin{equation}
\|   z_{\bar u,u-\bar u}\|_{L^2(Q)}\le \|z_{\bar u,u-\bar u}\|_{C(\bar Q)}^{\frac{2-s}{2}}\|u-\bar u \|  _{L^1(Q)}^{\frac{s}{2}}.
\label{helpful}
\end{equation}
By \eqref{helpful}, \eqref{clr}, \eqref{aeqestLs} and \eqref{E2.13L2}, we find
\begin{equation}
\begin{aligned}
\|y_u - \bar y\|_{C(\bar Q)} \|z_{\bar u,u - \bar u}\| _{L^2(Q)}^2&\leq \| y_u - \bar y\| _{C(\bar Q)}  \|z_{\bar u,u - \bar u}\|_{L^2(Q)} \| z_{\bar u,u - \bar u}\|   _{C(\bar Q)}^{(2-s)/2}\|u-\bar u\|_{L^1(Q)}^{s/2}\\
&\leq C_{s'}\sup_{\mathcal U} \| u - \bar u\|_{L^\infty(Q)}^{(s-1)/(s')}\|y_u - \bar y\|   _{C(\bar Q)}^{l_1+l_2}  \|    z_{\bar u,u - \bar u}\|_{L^2(Q)} \|u-\bar u\| _{L^1(Q)}^{(2-s)/(2s')+s/2}\\
&\leq C_{s'}^2\tilde M_{\mathcal U}\|y_u - \bar y\|_{C(\bar Q)}^{l_1} \| z_{\bar u,u - \bar u}\|   _{L^2(Q)}\|   u-\bar u\|   _{L^1(Q)}^{l_2/s'}  \| u-\bar u\|_{L^1(Q)}^{(2-s)/(2s')}\| u-\bar u\| _{L^1(Q)}^{s/2},
\end{aligned}
\label{Em.1}
\end{equation}
with $\tilde M:=M_{\mathcal U}^{\frac{s-1}{s'}(l_2+\frac{2-s}{2})}$. We select $l_2$ such that
\begin{equation*}
\frac{l_2}{s'}+\frac{2-s}{2s'}+\frac{s}{2}=1.
\end{equation*}
Using $1/s'=1-1/s$, this is equivalent to $(1+l_2)(1-1/s)+s/2(1-1+1/s)=1$,
thus we find 
\begin{equation*}
l_2=s'/2-1.
\end{equation*}
Defining $\varepsilon:=\frac{1}{C_{s'}^2\tilde M}\rho^{\frac{1}{l_1}}$ proves the first claim.
For the proof of \eqref{E3.17.2} we use \eqref{clr}, \eqref{aeqestLs} and \eqref{E2.13L2} to infer
\begin{equation}
\begin{aligned}
&\|   y_u - \bar y\|   _{C(\bar Q)} \|    z_{\bar u,v}\|   _{L^2(Q)}^2\leq C_{s'} \|   y_u - \bar y\|   _{C(\bar Q)} \|    z_{\bar u,v}\|   _{C(\bar Q)}^{(2-s)}\|    u-\bar u\|   _{L^1(Q)}^{s}\\
&\leq C_{s'}^2M_{\mathcal U}^{\frac{s'-1}{s'}}\|y_u - \bar y\|_{C(\bar Q)}^{l_1+l_2} \|u-\bar u\|_{L^1(Q)}^{(2-s)/s'}\|u-\bar u\|_{L^1(Q)}^{s}\\
&\leq C_{s'}^3 \tilde M\|   y_u - \bar y\|   _{C(\bar Q)}^{l_1}\|   u-\bar u\|   _{L^1(Q)}^{l_2/s'}  \|   u-\bar u\|   _{L^1(Q)}^{(2-s)/s'}\|    u-\bar u\|   _{L^1(Q)}^{s},
\end{aligned}
\label{Em.2}
\end{equation}
with $\tilde M:=M_{\mathcal U}^{\frac{s-1}{s'}(l_2+2-s)}$.
Select $l_2$ such that
\begin{equation*}
\frac{l_2}{s'}+\frac{2-s}{s'}+s=2.
\end{equation*}
By $\frac{1}{s'}=1-\frac{1}{s}$, this is equivalent to $l_2=\frac{2-s}{s-1}$.
Defining $\varepsilon:=\frac{1}{C_{s'}^3\tilde M}\rho^{\frac{1}{l_1}}$ proves the case for $m=0$. For $m\ne 0$, 
we recall, that the $L^1(Q)$-distance of the controls is assumed to be sufficiently small. But by the estimate \eqref{clr}, this implies 
that the states are close and we proceed as displayed.
\end{proof}

\bibliography{CJV}

\begin{thebibliography}{10}

\bibitem{ASS16}
Walter Alt, Christopher Schneider, and Martin Seydenschwanz.
\newblock Regularization and implicit {E}uler discretization of
  linear-quadratic optimal control problems with bang-bang solutions.
\newblock {\em Appl. Math. Comput.}, 287/288:104--124, 2016.

\bibitem{CDJ2022}
E.~Casas, A.~Dom\'{\i}nguez Corella, and N.~Jork.
\newblock New assumptions for stability analysis in elliptic optimal control
  problems.
\newblock {\em Submitted, Available at
  \url{https://orcos.tuwien.ac.at/research/research_reports/}}, 2022.

\bibitem{Casas-Mateos2020}
E.~Casas and M.~Mateos.
\newblock Critical cones for sufficient second order conditions in {PDE}
  constrained optimization.
\newblock {\em SIAM J. Optim.}, 30(1):585--603, 2020.

\bibitem{Casas12}
Eduardo Casas.
\newblock Second order analysis for bang-bang control problems of {PDE}s.
\newblock {\em SIAM J. Control Optim.}, 50(4):2355--2372, 2012.

\bibitem{CM2020}
Eduardo Casas and Mariano Mateos.
\newblock Critical cones for sufficient second order conditions in {PDE}
  constrained optimization.
\newblock {\em SIAM J. Optim.}, 30(1):585--603, 2020.

\bibitem{CM2021}
Eduardo Casas and Mariano Mateos.
\newblock State error estimates for the numerical approximation of sparse
  distributed control problems in the absence of {T}ikhonov regularization.
\newblock {\em Vietnam J. Math.}, 49(3):713--738, 2021.

\bibitem{CMCO}
Eduardo Casas and Mariano Mateos.
\newblock Corrigendum: {C}ritical cones for sufficient second order conditions
  in {PDE} constrained optimization.
\newblock {\em SIAM J. Optim.}, 32(1):319--320, 2022.

\bibitem{CMR}
Eduardo Casas, Mariano Mateos, and Arnd R\"{o}sch.
\newblock Error estimates for semilinear parabolic control problems in the
  absence of {T}ikhonov term.
\newblock {\em SIAM J. Control Optim.}, 57(4):2515--2540, 2019.

\bibitem{CRT15}
Eduardo Casas, Christopher Ryll, and Fredi Tr\"{o}ltzsch.
\newblock Second order and stability analysis for optimal sparse control of the
  {F}itz{H}ugh-{N}agumo equation.
\newblock {\em SIAM J. Control Optim.}, 53(4):2168--2202, 2015.

\bibitem{CT16}
Eduardo Casas and Fredi Tr\"{o}ltzsch.
\newblock Second-order optimality conditions for weak and strong local
  solutions of parabolic optimal control problems.
\newblock {\em Vietnam J. Math.}, 44(1):181--202, 2016.

\bibitem{CT22}
Eduardo Casas and Fredi Tr\"{o}ltzsch.
\newblock Stability for semilinear parabolic optimal control problems with
  respect to initial data.
\newblock {\em Appl. Math. Optim.}, 86(16), 2022.

\bibitem{CWW}
Eduardo Casas, Daniel Wachsmuth, and Gerd Wachsmuth.
\newblock Sufficient second-order conditions for bang-bang control problems.
\newblock {\em SIAM J. Control Optim.}, 55(5):3066--3090, 2017.

\bibitem{CWW2018}
Eduardo Casas, Daniel Wachsmuth, and Gerd Wachsmuth.
\newblock Second-order analysis and numerical approximation for bang-bang
  bilinear control problems.
\newblock {\em SIAM J. Control Optim.}, 56(6):4203--4227, 2018.

\bibitem{Chipot}
Michel Chipot.
\newblock {\em Elements of nonlinear analysis}.
\newblock Birkh\"{a}user Advanced Texts: Basler Lehrb\"{u}cher. [Birkh\"{a}user
  Advanced Texts: Basel Textbooks]. Birkh\"{a}user Verlag, Basel, 2000.

\bibitem{CDK}
R.~Cibulka, A.~L. Dontchev, and A.~Y. Kruger.
\newblock Strong metric subregularity of mappings in variational analysis and
  optimization.
\newblock {\em J. Math. Anal. Appl.}, 457(2):1247--1282, 2018.

\bibitem{CDKV-17}
R~Cibulka, A.L. Dontchev, and V.M. Veliov.
\newblock Metrically regular differential generalized equations.
\newblock {\em SIAM J. Control Optim.}, 56(1):316--342, 2018.

\bibitem{DJV2022}
A.~Dom\'{\i}nguez Corella, N.~Jork, and V.~Veliov.
\newblock Stability in affine optimal control problems constrained by
  semilinear elliptic partial differential equations.
\newblock {\em Submitted, Available at
  \url{https://orcos.tuwien.ac.at/research/research_reports/}}, 2022.

\bibitem{DR}
Asen~L. Dontchev and R.~Tyrrell Rockafellar.
\newblock {\em Implicit functions and solution mappings}.
\newblock Springer Monographs in Mathematics. Springer, Dordrecht, 2009.
\newblock A view from variational analysis.

\bibitem{Dunn1998}
J.~C. Dunn.
\newblock On second order sufficient conditions for structured nonlinear
  programs in infinite-dimensional function spaces.
\newblock In {\em Mathematical programming with data perturbations}, volume 195
  of {\em Lecture Notes in Pure and Appl. Math.}, pages 83--107. Dekker, New
  York, 1998.

\bibitem{Evans}
Lawrence~C. Evans.
\newblock {\em Partial differential equations}, volume~19 of {\em Graduate
  Studies in Mathematics}.
\newblock American Mathematical Society, Providence, RI, second edition, 2010.

\bibitem{MZ1979}
H.~Maurer and J.~Zowe.
\newblock First and second order necessary and sufficient optimality conditions
  for infinite-dimensional programming problems.
\newblock {\em Math. Programming}, 16(1):98--110, 1979.

\bibitem{OSVE2}
N.~P. Osmolovskii and V.~M. Veliov.
\newblock Metric sub-regularity in optimal control of affine problems with free
  end state.
\newblock {\em ESAIM Control Optim. Calc. Var.}, 26:Paper No. 47, 19, 2020.

\bibitem{OSVE}
Nikolai~P. Osmolovskii and Vladimir~M. Veliov.
\newblock On the regularity of {M}ayer-type affine optimal control problems.
\newblock In {\em Large-scale scientific computing}, volume 11958 of {\em
  Lecture Notes in Comput. Sci.}, pages 56--63. Springer, Cham, [2020]
  \copyright 2020.

\bibitem{PSS}
J.-S Pang and D.A. Steward.
\newblock Differential variational inequalities.
\newblock {\em Math. Programming A}, 116(1):345--424, 2008.

\bibitem{Qui-Wachsmuth2018}
N.~T. Qui and D.~Wachsmuth.
\newblock Stability for bang-bang control problems of partial differential
  equations.
\newblock {\em Optimization}, 67(12):2157--2177, 2018.

\bibitem{S2015}
Martin Seydenschwanz.
\newblock Convergence results for the discrete regularization of
  linear-quadratic control problems with bang-bang solutions.
\newblock {\em Comput. Optim. Appl.}, 61(3):731--760, 2015.

\bibitem{Showalter}
R.~E. Showalter.
\newblock {\em Monotone operators in {B}anach space and nonlinear partial
  differential equations}, volume~49 of {\em Mathematical Surveys and
  Monographs}.
\newblock American Mathematical Society, Providence, RI, 1997.

\bibitem{Troltzsch2010}
F.~{Tr\"{o}ltzsch}.
\newblock {\em Optimal Control of Partial Differential Equations: Theory,
  Methods and Applications}, volume 112 of {\em Graduate Studies in
  Mathematics}.
\newblock American Mathematical Society, Philadelphia, 2010.

\bibitem{VD18}
Nikolaus von Daniels.
\newblock Tikhonov regularization of control-constrained optimal control
  problems.
\newblock {\em Comput. Optim. Appl.}, 70(1):295--320, 2018.

\bibitem{V03}
Ioan~I. Vrabie.
\newblock {\em {$C_0$}-semigroups and applications}, volume 191 of {\em
  North-Holland Mathematics Studies}.
\newblock North-Holland Publishing Co., Amsterdam, 2003.

\bibitem{Wloka}
J.~Wloka.
\newblock {\em Partial differential equations}.
\newblock Cambridge University Press, Cambridge, 1987.
\newblock Translated from the German by C. B. Thomas and M. J. Thomas.

\end{thebibliography}
\bibliographystyle{plain}

\end{document}